\documentclass[a4paper,12pt,oneside]{amsart}
               
\usepackage{a4wide}
\usepackage[utf8]{inputenc} 
\usepackage[T1]{fontenc} 
\usepackage{textcomp} 
\usepackage[english,french]{babel} 
\usepackage{amsmath}
\usepackage{amssymb}
\usepackage{amsthm}
\usepackage{graphicx}
\usepackage{url}
\usepackage{hyperref}
\usepackage[all]{xy}
\usepackage{nccmath}
\usepackage{tikz-cd}
\usepackage{color}
\usepackage{mathrsfs,xspace}
\usepackage{comment}
\usepackage{extarrows}
\usepackage{enumitem}
\usepackage[foot]{amsaddr}
\usepackage{faktor}
\usepackage[font=small,labelfont=bf]{caption}
\allowdisplaybreaks

\DeclareMathOperator{\supp}{supp}

\DeclareMathOperator{\codim}{codim}

\DeclareMathOperator{\Sub}{sub}

\DeclareMathOperator{\Op}{Op}

\DeclareMathOperator{\dom}{dom}
\DeclareMathOperator{\Mod}{Mod}
\DeclareMathOperator{\Db}{\mathcal{D}\textit{b}}
\DeclareMathOperator{\DR}{\mathcal{DR}}

\DeclareMathOperator{\Sol}{\mathcal{S}\textit{ol}}
\DeclareMathOperator{\N}{\mathbb{N}}
\DeclareMathOperator{\Z}{\mathbb{Z}}
\DeclareMathOperator{\R}{\mathbb{R}}
\DeclareMathOperator{\RR}{R}
\DeclareMathOperator{\EE}{E}
\DeclareMathOperator{\DD}{D}

\DeclareMathOperator{\bb}{b}
\DeclareMathOperator{\ttt}{t}
\DeclareMathOperator{\TT}{T}
\DeclareMathOperator{\LLL}{L}
\DeclareMathOperator{\PP}{P}
\DeclareMathOperator{\PPP}{\mathbb{P}}
\DeclareMathOperator{\V}{\mathbb{V}}
\DeclareMathOperator{\C}{\mathbb{C}}

\DeclareMathOperator{\Hol}{\mathcal{O}}

\DeclareMathOperator{\CC}{\mathcal{C}}
\DeclareMathOperator{\LL}{\mathcal{L}}

\DeclareMathOperator{\RHom}{RHom}
\DeclareMathOperator{\Conv}{Conv}
\DeclareMathOperator{\qgood}{q-good}
\DeclareMathOperator{\hol}{hol}
\DeclareMathOperator{\Exp}{Exp}

\theoremstyle{plain}
\newtheorem{theorem}{Theorem}[section]
\newtheorem{corollary}[theorem]{Corollary}
\newtheorem{lemma}[theorem]{Lemma}
\newtheorem{proposition}[theorem]{Proposition}

\theoremstyle{definition}
\newtheorem{definition}[theorem]{Definition}

\theoremstyle{remark}
\newtheorem{remark} [theorem]{Remark}

\numberwithin{equation}{section}

\begin{document}
\selectlanguage{english}

\title{Enhanced Laplace transform and holomorphic Paley-Wiener-type theorems}
\author{Christophe Dubussy}
\date{\today}
\address{B\^at. B37 \\ Analyse alg\'{e}brique \\ Quartier Polytech 1 \\ All\'{e}e de la d\'{e}couverte 12 \\ 4000 Li\`{e}ge \\ Belgique.}
\email{C.Dubussy@uliege.be}
\thanks{We were supported by a FNRS grant (ASP 111496F)}

\begin{abstract}
Starting from a remark about the computation of Kashiwara-Schapira's enhanced Laplace transform by using the Dolbeault complex of enhanced distributions, we explain how to obtain explicit holomorphic Paley-Wiener-type theorems. As an example, we get back some classical theorems due to Polya and M\'{e}ril as limits of tempered Laplace-isomorphisms. In particular, we show how contour integrations naturally appear in this framework. 
\end{abstract}

\maketitle
\vspace{-1em}
\section{Introduction}

In \cite{Dagn16}, A. D'Agnolo and M. Kashiwara extended the Riemann-Hilbert correspondence to cover the case of holonomic $\mathcal{D}$-modules with irregular singularities. This progress allows to deal with integral transforms which have an irregular kernel, such as the Laplace transform on a complex vector space $\V$. A first work on this subject was done in \cite{Dagn14}, where A. D'Agnolo studied the Laplace transform in the non-conic case, extending the results of \cite{Kash97}. In particular, he explained how this sheaf-theoretic transformation allows to get back some classical real Paley-Wiener theorems.

\medskip

More recently, in \cite{Kash16}, M. Kashiwara and P. Schapira made a full rewriting of the theory of integral transforms with irregular kernel, using the notion of enhanced ind-sheaves introduced in \cite{Dagn16}. In particular, they treated the case of the Laplace transform. 

\medskip

Let $\V$ be a $n$-dimensional complex vector space and $\V^*$ its complex dual. Let us consider the bordered spaces $\V_{\infty} = (\V,\overline{\V})$ and $\V^*_{\infty} = (\V^*,\overline{\V}^*)$ where $\overline{\V}$ (resp. $\overline{\V}^*$) is the projective compactification of $\V$ (resp $\V^*$). In \cite{Kash16}, the authors proved that there is a canonical isomorphism 
\begin{equation}\label{equ:isointroduction}
{}^{\EE}\!\mathcal{F}^a_{\V}(\Omega^{\EE}_{\V_{\infty}})[n] \simeq \Hol^{\EE}_{\V^*_{\infty}}
\end{equation}
in $\EE^{\bb}(\C^{\Sub}_{\V^*_{\infty}})$, where ${}^{\EE}\!\mathcal{F}^a_{\V}$ is the enhanced Fourier-Sato functor and $\Omega^{\EE}_{\V_{\infty}}$ (resp. $\Hol^{\EE}_{\V^*_{\infty}}$) is the complex of enhanced holomorphic top-forms on $\V_{\infty}$ (resp. enhanced holomorphic functions on $\V^*_{\infty}$). 

\medskip
The first part of our paper consists in a remark about $(\ref{equ:isointroduction})$. Using the Dolbeault complex $\Db^{T, \bullet, \bullet}$ of enhanced distributions, we show that there is a canonical morphism
\begin{equation}\label{equ:explicitintroduction}
{q}_{\R !!} ({\mu_{-\langle z,w \rangle}}_* p_{\R}^{-1}\Db^{\TT,n,\bullet+n}_{\V_{\infty}}) \to \Db^{\TT,0,\bullet}_{\V^*_{\infty}},
\end{equation}
where $p : \V_{\infty} \times \V^*_{\infty} \to \V_{\infty}$ and $q : \V_{\infty} \times \V^*_{\infty} \to \V^*_{\infty}$ are the two projections and $\mu_{-\langle z,w \rangle}$ is the translation by $-\langle z,w \rangle$. This morphism encodes the usual positive Laplace transform of distributions and is equivalent to~(\ref{equ:isointroduction}) in $\EE^{\bb}(\C^{\Sub}_{\V^*_{\infty}})$. In order to prove that, we had to trace back all the steps in the construction of~(\ref{equ:isointroduction}), which led to several morphisms defined in \cite{Kash84}, \cite{Kash96} and \cite{Kash01}. Of course, we shall only present a sketch of this historical compilation and refer to the relevant articles when needed. 

\medskip

This remark has an immediate application. Let $f : \V \to \R$ be a continuous function and $S$ be a subanalytic closed subset of $\V$. Let us denote by $f_S$ the function which is equal to $f$ on $S$ and to $+\infty$ on $\V \backslash S$ and assume that $f_S$ is convex. Under suitable conditions, we shall show that there is a commutative diagram

\[
\xymatrix{H^n_{S}(\V, e^{-f}\Omega^{\ttt}_{\overline{\V}}) \ar[r]^{\sim} & H^0(\V^*, e^{f_S^*}\Hol^{\ttt}_{\overline{\V}^*}) \ar[d] \\
\Gamma(\V, \mathscr{I}\Gamma_{S}(e^{-f}\Db^{\ttt, n,n}_{\overline{\V}}))\ar[r]\ar[u] & \Gamma(\V^*, e^{f_S^*}\Db^{\ttt}_{\overline{\V}^*})}
\] 
where $f_S^*$ is the Legendre transform of $f_S$ and $\Db^{t,\bullet,\bullet}$ (resp. $\Omega^{\ttt}, \Hol^{\ttt}$) is the Dolbeault complex of tempered distributions  (resp. complex of tempered holomorphic top-forms, functions). Here, the top isomorphism comes from \cite{Kash16} and the bottom one is given by the positive Laplace transform of distributions. 

\medskip
The main part of the paper consists in explaining how this diagram allows to obtain holomorphic Paley-Wiener-type theorems. We shall give two examples. Let $\V$ be of dimension $1$, let $K$ be a non-empty convex compact subset of $\V$ and let $h_K : w \in \V^* \mapsto \sup_{z\in K} \Re \langle z,w \rangle$ be its support function. Under the same suitable conditions as above, we prove that, for all $\varepsilon>0,$ there is a canonical isomorphism of $\C$-vector spaces
\begin{equation}
\Omega^{\ttt}_{\PPP}(\V\backslash K_{\varepsilon})/\Omega^{\ttt}_{\PPP}(\V) \xrightarrow{\sim} e^{h_{K_{\varepsilon}}}\Hol^{\ttt}_{\PPP^*}(\V^*),
\end{equation}
where $\PPP = \V \cup \{\infty\}$ and $\PPP^* = \V^* \cup \{\infty\}$. Taking the projective limit on $\varepsilon \to 0$, we get an isomorphism
\begin{equation}
\Omega^{\ttt_{\infty}}_{\PPP}(\V\backslash K)/\Omega^{\ttt}_{\PPP}(\V) \xrightarrow{\sim} \varprojlim_{\varepsilon \to 0} e^{h_{K_{\varepsilon}}}\Hol^{\ttt}_{\PPP^*}(\V^*),
\end{equation}
\noindent
where $\Omega^{\ttt_{\infty}}_{\PPP}$ is the sheaf of holomorphic forms which are tempered only at infinity. Given a global $\C$-linear coordinate $z$ of $\V$ and $w$ its dual coordinate, we show that this last isomorphism can be explicited by $[u(z)dz] \mapsto v$, with
\[
v(w) = \int_{C(0,r)^+} e^{zw} u(z)dz,
\] 
where $C(0,r)^+$ is a positively oriented circle, which encloses $K$. This result is actually equivalent to an old theorem due to Polya (see e.g. \cite{Bere95}). However, in our approach, it is obtained as a limit of tempered Laplace-isomorphisms and, moreover, the contour integration naturally appears as a cohomological phenomenon. In the last section, we prove a non-compact analogue of this theorem, which is due to Méril (see \cite{Meril83}). This result is also obtained as a limit of tempered Laplace-isomorphisms.

\section{Background and notations}

\paragraph{2.1}
We refer to \cite{Kash90} for sheaf theory and derived categories and to \cite{Kash03} for $\mathscr{D}$-modules. Let $X$ be a complex manifold and $Y \subset X$ be a complex analytic hypersurface. One denotes by $\Hol_X(*Y)$ the sheaf of holomorphic functions with poles in $Y$. For any $\varphi \in \Hol_X(*Y)$, one sets 
\[
\mathscr{D}_X e^{\varphi} = \mathscr{D}_X / \{P : Pe^{\varphi} = 0 \,\,\text{on}\,\, U\} \qquad \text{and} \qquad \mathscr{E}^{\varphi}_{U|X} =\mathscr{D}_X e^{\varphi}  \overset{\DD}\otimes \Hol_X(*Y),
\] 
where $U = X \backslash Y.$

\medskip

\paragraph{2.2}
We refer to \cite{Kashi16}, \cite{Kash01} and \cite{Prel08} for subanalytic spaces and subanalytic sheaves. Let $M$ be a subanalytic space. We write for short $\text{Op}_M^{\Sub, c}$ the set of open subanalytic subsets of $M$ which are relatively compact. We denote by $M^{\Sub}$ the associated subanalytic site and by $\rho : M \to M^{\Sub}$ the canonical morphism of sites. This morphism induces three functors 

\[
\xymatrixrowsep{1in}
\xymatrixcolsep{2in}
\xymatrix{\text{Mod}(\C_M) \ar@/^1pc/[r]^{\rho_*} \ar@/_1pc/[r]_{\rho_!} & \text{Mod}(\C^{\Sub}_M) \ar[l]|{\rho^{-1}}}
\]

\noindent
between sheaves and subanalytic sheaves. In the following sections, the functor $\rho_!$ normaly occurs in many formulas involving $\mathscr{D}$-modules. In order to have less complicated formulas, we shall not write this functor, as in \cite{Kash16}. The derived category $\DD^{\bb}(\C^{\Sub}_M) :=\DD^{\bb}(\text{Mod}(\C^{\Sub}_M))$ has six Grothendieck operations, namely ${\RR}f_*, f^{-1}, {\RR}f_{!!}, f^!, {\RR}\mathscr{I}\!hom$ and $\overset{\LLL}\otimes.$ Let $Z$ be a subanalytic locally closed subset of $M$. One sets 
$
\mathscr{I}\Gamma_Z(-) = {\RR}\mathscr{I}\!hom(\C_Z,-).
$
We set for short $H^k_Z(U,-) = H^k{\RR}\Gamma(U,\mathscr{I}\Gamma_Z(-))$ for any $U \in \text{Op}_M^{\Sub, c}$.

\medskip

\paragraph{2.3}

Let $M$ be a real analytic manifold. For each $r\in\Z,$ let us denote denote by $\Db_M^r$ the sheaf or $r$-distributional forms. One defines a subanalytic sheaf $\Db^{\ttt, r}_M$ by setting $$\Gamma(U, \Db^{\ttt, r}_M) = \{\omega \in \Db_M^r(U) : \omega \,\, \text{can be extended to} \, \, M\}$$ for any $U \in \text{Op}_M^{\Sub, c}$. By construction, this sheaf is quasi-injective. Other classical definitions of $\Db^{\ttt, r}_M$ can be found in \cite{Kash84}. The sheaf $\Db^{\ttt}_M := \Db^{\ttt, 0}_M$ is called the sheaf of tempered distributions. One also introduces the subanalytic sheaf $\CC^{\infty, \ttt}_M$ of $\CC^{\infty}$-tempered functions by defining $\Gamma(U,\CC^{\infty, \ttt}_M)$ as the set of functions which have, as well as all their derivatives, polynomial growth near the boundary of $U.$ 

\medskip
\noindent
On a complex manifold $X$ of complex dimension $d_X$, one denotes by $\Omega^{\ttt,p}_X \in \DD^{\bb}(\C^{\Sub}_X)$ the Dolbeault complex 
\[
0 \to \Db_X^{\ttt, p,0} \overset{\bar{\partial}}\to \Db_X^{\ttt, p,1} \to \cdots \to \Db_X^{\ttt, p,d_X} \to 0.
\] 
One sets for short $\Hol_X^{\ttt} := \Omega^{\ttt,0}_X$ and $\Omega^{\ttt}_X :=\Omega^{\ttt,d_X}_X.$ 

\medskip

\paragraph{2.4}

We refer to \cite{Dagn16} for bordered spaces and to \cite{Kashi16} for subanalytic sheaves on subanalytic bordered spaces. A bordered space is a couple $M_{\infty} = (M, \widehat{M})$ where $\widehat{M}$ is a good topological space and $M$ an open subset of $\widehat{M}.$ If $M_{\infty}=(M,\widehat{M})$ and $N_{\infty}=(N,\widehat{N})$ are two bordered spaces and if $f : M \to N$ is a coutinuous map, we denote by $\Gamma_f \subset M \times N$ the graph of $f$ and by $\overline{\Gamma}_f$ the closure of $\Gamma_f$ in $\widehat{M} \times \widehat{N}.$ A morphism of bordered spaces $f : M_{\infty} \to N_{\infty}$ is a continuous map $f : M \to N$ such that the canonical projection $\overline{\Gamma}_f \to \widehat{M}$ is proper. Such a morphism is called semi-proper if $\overline{\Gamma}_f \to \widehat{N}$ is proper. The composition of two morphisms is the composition of the underlying continuous maps. If $\widehat{N}$ is compact, then any continuous map $f : M \to N$ is a morphism of bordered spaces. On a subanalytic bordered space $M_{\infty}=(M,\widehat{M})$, we write for short $\text{Op}_{M_{\infty}}^{\Sub, c}$ the set of open subanalytic subsets of $M$ which are relatively compact in $\widehat{M}$ and $\Mod(\C_{M_{\infty}}^{\Sub})$ the category of subanalytic sheaves on $M_{\infty}.$ The six Grothendieck operations extend to this framework.

\medskip

\paragraph{2.5}
We write $\overline{\R} := \R \sqcup \{-\infty, +\infty\}$ the 2-points compatification of $\R$ and we consider the bordered space $\R_{\infty} = (\R, \overline{\R}).$  Let $M_{\infty} = (M,\widehat{M})$ be a subanalytic bordered space. There is a well-defined convolution functor
\[
- \overset{+}\otimes - : \DD^{\bb}(\C^{\Sub}_{M_{\infty}\times \R_{\infty}}) \times \DD^{\bb}(\C^{\Sub}_{M_{\infty}\times \R_{\infty}}) \to \DD^{\bb}(\C^{\Sub}_{M_{\infty}\times \R_{\infty}}),
\] 
which has a right adjoint $\mathscr{I}\!hom^{+}(-,-)$. If $\varphi : M \to \R$ is a continuous function, we denote by $\mu_{\varphi} : M_{\infty} \times \R_{\infty} \to M_{\infty} \times \R_{\infty}$ the map defined by $\mu_{\varphi}(x,t) = (x, t + \varphi(x)).$ Let us recall that $\C_{\{t=\varphi(x)\}} \overset{+}\otimes F \simeq R{\mu_{\varphi}}_* F$ for any $F \in \DD^{\bb}(\C^{\Sub}_{M_{\infty}\times \R_{\infty}}).$ 

\medskip

\paragraph{2.6}
We refer to \cite{Dagn16} for enhanced ind-sheaves (or similarly enhanced subanalytic sheaves) and to \cite{Kashi16} for a pedagogical exposition. On a subanalytic bordered space $M_{\infty}=(M,\widehat{M})$, one defines the category of enhanced subanalytic sheaves by setting 
\[
{\EE}^{\bb}(\C^{\Sub}_{M_{\infty}}) = \DD^{\bb}(\C^{\Sub}_{M_{\infty}\times \R_{\infty}})/\{F : (\C_{\{t\geq 0\}} \oplus \C_{\{t \leq 0\}}) \overset{+}\otimes F \simeq 0\}.
\] 
We denote by $$Q_{M_{\infty}} : \DD^{\bb}(\C^{\Sub}_{M_{\infty}\times \R_{\infty}}) \to {\EE}^{\bb}(\C^{\Sub}_{M_{\infty}})$$ the quotient functor. If the context is clear, we shall simply write $\C_{A}$ instead of $Q_{M_{\infty}}(\C_A)$ when $A$ is a subanalytic locally closed subset of $M \times \R.$ Recall that there is a well-defined hom-functor
\[
{\RR}\mathscr{I}\!hom^{\EE}(-,-) : \EE^{\bb}(\C^{\Sub}_{M_{\infty}})^{\text{op}} \times \EE^{\bb}(\C^{\Sub}_{M_{\infty}}) \to \DD^{\bb}(\C^{\Sub}_{M_{\infty}}).
\]
One sets
\[
{\RR}\text{Hom}^{\EE}(F_1,F_2) = {\RR}\Gamma(M, \rho^{-1}{\RR}\mathscr{I}\!hom^{\EE}(F_1,F_2)),
\]
for all $F_1,F_2 \in {\EE}^{\bb}(\C^{\Sub}_{M_{\infty}}).$ Moreover, if $f : M_{\infty} \to N_{\infty}$ is a morphism of subanalytic bordered spaces, one sets $$f_{\R} := f\times \text{id}_{\R} : M_{\infty} \times \R_{\infty} \to N_{\infty} \times \R_{\infty}$$ and one writes ${\EE}f_{*}, {\EE}f^{-1}\hspace{-0.2em}, {\EE}f_{!!}$ and ${\EE}f^!$ the functors which are the factorisations of ${\RR}{f_{\R}}_*, f_{\R}^{-1}$, ${\RR}{f_{\R}}_{!!}$ and $f_{\R}^!$ through $Q_{M_{\infty}}$ and $Q_{N_{\infty}}$. Finally, the convolution functors $\overset{+}\otimes, \mathscr{I}\!hom^{+}$ also factor through the quotient and we keep the same notations for their factorisation. 

\medskip

\paragraph{2.7} 
Assume that $\widehat{M}$ is real analytic. Let $\PP$ be the projective compactification of $\R$ and let $j : M_{\infty} \times \R_{\infty} \to \widehat{M} \times \PP$ be the canonical inclusion. Let $t$ be the affine coordinate of $\PP$. Then $\partial_t$ extends to a vector field on $\PP$ and, for each $r\in \Z$, one sets 
\[
\Db^{\TT, r}_{M_{\infty}} = j^{-1}( \ker(\Db^{\ttt,r}_{\widehat{M}\times \PP} \overset{\partial_t -1}\longrightarrow \Db^{\ttt,r}_{\widehat{M}\times \PP})) \in \Mod(\C^{\Sub}_{M_{\infty}\times \R_{\infty}}).
\] 
One calls $\Db^{\TT}_{M_{\infty}} :=\Db^{\TT, 0}_{M_{\infty}}$ the sheaf of enhanced distributions and one sets for short $\Db^{\EE, r}_{M_{\infty}} = Q_{M_{\infty}}(\Db^{\TT, r}_{M_{\infty}})$ (see \cite{Dagn16} and \cite{Kashi16}).

\medskip

\paragraph{2.8}
The notion of complex bordered space is similarly defined in section $4.3$ of \cite{Kash16}. On a complex bordered space $X_{\infty}$, one denotes by $\DD^{\bb}(\mathscr{D}_{X_{\infty}})$ (resp. $\DD^{\bb}_{\text{hol}}(\mathscr{D}_{X_{\infty}})$ and $\DD^{\bb}_{\text{q-good}}(\mathscr{D}_{X_{\infty}})$) the bounded derived categories of $\mathscr{D}$-modules (resp. holonomic $\mathscr{D}$-modules and quasi-good $\mathscr{D}$-modules) over $X_{\infty}.$ Moreover, the usual operations of $\mathscr{D}$-modules naturally extend to this framework. If $f : X_{\infty}=(X, \widehat{X}) \to Y_{\infty}=(Y, \widehat{Y})$ is a morphism of complex bordered spaces, one notes $\mathscr{D}_{X_{\infty}\to Y_{\infty}}$ the associated transfer bi-module. On a complex bordered space $X_{\infty}$, one defines $\Omega^{\EE,p}_{X_{\infty}}$
by the Dolbeault complex 
\[
Q_{X_{\infty}}(0\to \Db^{\TT,p,0}_{X_{\infty}} \overset{\bar{\partial}}\to \Db^{\TT,p,1}_{X_{\infty}} \to \dots \to \Db^{\TT,p,d_X}_{X_{\infty}} \to 0).
\] 
One sets for short 
$
\Hol^{\EE}_{X_{\infty}} := \Omega^{\EE,0}_{X_{\infty}}$ and $\Omega^{\EE}_{X_{\infty}} := \Omega^{\EE,d_X}_{X_{\infty}}.$ Finally, one also defines the enhanced de Rham and solution functors by setting 
\begin{align*}
\mathcal{D}\mathcal{R}^{\EE}_{X_{\infty}} &: \DD^{\bb}(\mathscr{D}_{X_{\infty}}) \to {\EE}^{\bb}(\C^{\Sub}_{X_{\infty}}), \quad \mathcal{M}\mapsto \Omega^{\EE}_{X_{\infty}} \overset{\LLL}\otimes_{\mathscr{D}_{X_{\infty}}} \mathcal{M}, \\
\Sol^{\:\EE}_{X_{\infty}} &: \DD^{\bb}(\mathscr{D}_{X_{\infty}})^{\text{op}} \to {\EE}^{\bb}(\C^{\Sub}_{X_{\infty}}), \quad \mathcal{M} \mapsto {\RR}\mathcal{H}om_{\mathscr{D}_{X_{\infty}}} (\mathcal{M}, \Hol^{\EE}_{X_{\infty}}).
\end{align*}

\section{Integration and pullback of tempered distributions}

Integration and pullback of distributional forms are very classical constructions in differential calculus (see e.g. \cite{Rham84}). In this short section, we recall some results of \cite{Kash84}, \cite{Kash96} and \cite{Kash01} about integration and pullback of tempered distributional forms. 

\begin{proposition}[\cite{Kash96}, Proposition $4.3$ and Theorem $5.7$]\label{prop:tempintegration}
Let $f : X \to Y$ be a holomorphic map between complex manifolds. The integration of distributions along the fibers of $f$ induces a morphism of double complexes
\begin{equation}
\int_f : f_{!!} \Db^{\ttt, \bullet+d_X, \bullet+d_X}_{X} \rightarrow \Db^{\ttt, \bullet+d_Y, \bullet+d_Y}_{Y}
\end{equation}
and thus, a morphism
\begin{equation}\label{equ:tempintegration}
\int_f : {\RR} f_{!!} \Omega^{\ttt,p+d_X}_{X} [d_X] \rightarrow \Omega^{\ttt,p+d_Y}_{Y} [d_Y]
\end{equation}
\noindent in ${\DD}^{\bb}(\C^{\Sub}_{Y})$, for each $p\in \Z.$ 
\end{proposition}

\begin{proposition}[\cite{Kash01}, Lemmas $7.4.4$ and $7.4.5$]\label{prop:tempintegration2}
Let $f : X \to Y$ be a holomorphic map between complex manifolds. There is a natural isomorphism
\begin{equation}\label{equ:modadjointtemp}
\Omega^{\ttt}_{X} \overset{\LLL}\otimes_{\mathscr{D}_{X}} \mathscr{D}_{X \to Y} [d_X] \xrightarrow{\sim} f^{!} \Omega^{\ttt}_{Y} [d_Y]
\end{equation}
in $\DD^{\bb}(\C^{\Sub}_{X}).$ Its adjoint morphism 
\[
{\RR}f_{!!} (\Omega^{\ttt}_{X} \overset{\LLL}\otimes_{\mathscr{D}_{X}} \mathscr{D}_{X \to Y})[d_X] \rightarrow \Omega^{\ttt}_{Y} [d_Y]
\]
induces, thanks to the canonical section $1_{X\to Y}$ of $\mathscr{D}_{X \to Y}$, a morphism 
\[
{\RR} f_{!!} \Omega^{\ttt}_{X} [d_X] \rightarrow \Omega^{\ttt}_{Y} [d_Y],
\]
which is equivalent to~(\ref{equ:tempintegration}) when $p=0.$
\end{proposition}

\begin{proposition}[\cite{Kash84}, Proposition $3.9$]
Let $f : X \to Y$ be a submersive holomorphic map between complex manifolds. The pullback of distributions by $f$ induces a morphism of double complexes
\begin{equation}
f^* : f^{-1} \Db^{\ttt, \bullet, \bullet}_Y \rightarrow \Db^{\ttt, \bullet, \bullet}_X
\end{equation}
and thus, a morphism
\begin{equation}\label{equ:temppullback}
f^* : f^{-1} \Omega_{Y}^{\ttt,p} \rightarrow \Omega_{X}^{\ttt,p}
\end{equation}
\noindent 
in ${\DD}^{\bb}(\C^{\Sub}_{X})$, for each $p\in \Z.$ 
\end{proposition}

\begin{proposition}[\cite{Kash96}, Theorems $4.5$, $5.8$ and \cite{Kash01}, Lemma $7.4.9$]
Let $f : X \to Y$ be a holomorphic map between complex manifolds. There is a natural morphism
\begin{equation}\label{equ:modpullback}
\mathscr{D}_{X \to Y} \otimes_{f^{-1}\mathscr{D}_{Y}} f^{-1}\Hol^{\ttt}_Y \rightarrow \Hol^{\ttt}_X.
\end{equation}
in $\DD^{\bb}(\C^{\Sub}_{X}).$ The canonical section $1_{X\to Y}$ induces a morphism
\[
f^{-1}\Hol^{\ttt}_Y \rightarrow \Hol^{\ttt}_X,
\]
which is equivalent to~(\ref{equ:temppullback}) when $p=0,$ if $f$ is submersive.
\end{proposition}

\section{Operations on enhanced distributions}

In this section, we extend the previous constructions to enhanced distributional forms on bordered spaces and we also treat the case of the multiplication by an exponential factor. Using Dolbeault resolutions, we point out that these constructions are encoded in the important results of \cite{Kash16}.

\begin{lemma}\label{lem:acyclic}
Let $f : M_{\infty} = (M,\widehat{M}) \to N_{\infty} = (N,\widehat{N})$ be a morphism of real analytic bordered spaces. The sheaf $\Db^{\TT}_{M_{\infty}}$ is acyclic for $f_{\R}{}_{!!}.$
\end{lemma}
\begin{proof}
By Lemma $6.2.4$ in \cite{Dagn16}, we get a short exact sequence  
\[
0 \longrightarrow \ker(\Db^{\ttt}_{\widehat{M}\times \PP} \overset{\partial_t -1}\longrightarrow \Db^{\ttt}_{\widehat{M}\times \PP}) \longrightarrow \Db^{\ttt}_{\widehat{M}\times \PP} \overset{\partial_t -1}\longrightarrow \Db^{\ttt}_{\widehat{M}\times \PP} \longrightarrow 0.
\]
Let $j : M_{\infty} \times \R_{\infty} \to \widehat{M} \times \PP$ be the canonical inclusion of bordered spaces. Applying $j^{-1}$ to the previous sequence, we get a short exact sequence 
\[
0 \longrightarrow \Db^{\TT}_{M_{\infty}} \longrightarrow j^{-1}\Db^{\ttt}_{\widehat{M}\times \PP} \overset{\partial_t -1}\longrightarrow j^{-1}\Db^{\ttt}_{\widehat{M}\times \PP} \longrightarrow 0
\]
and thus, a long exact sequence
\begin{center}
\begin{tikzcd}[row sep=large, column sep=2.2ex]
 0 \to f_{\R}{}_{!!}\Db^{\TT}_{M_{\infty}} \rar & f_{\R}{}_{!!}(j^{-1}\Db^{\ttt}_{\widehat{M}\times \PP}) \rar
             \ar[draw=none]{d}[name=X, anchor=center]{}
    & f_{\R}{}_{!!}(j^{-1}\Db^{\ttt}_{\widehat{M}\times \PP}) \ar[rounded corners,
            to path={ -- ([xshift=2ex]\tikztostart.east)
                      |- (X.center) \tikztonodes
                      -| ([xshift=-2ex]\tikztotarget.west)
                      -- (\tikztotarget)}]{dll}[at end]{} \\      
   {\RR^1}f_{\R}{}_{!!}\Db^{\TT}_{M_{\infty}} \rar & {\RR^1}f_{\R}{}_{!!}(j^{-1}\Db^{\ttt}_{\widehat{M}\times \PP}) \rar & {\RR^1}f_{\R}{}_{!!}(j^{-1}\Db^{\ttt}_{\widehat{M}\times \PP}) \to \dots
\end{tikzcd}
\end{center}
\noindent
Since $\Db^{\ttt}_{\widehat{M}\times \PP}$ is quasi-injective, we know that ${\RR}^kf_{\R}{}_{!!}(j^{-1}\Db^{\ttt}_{\widehat{M}\times \PP}) \simeq 0$ for all $k\geq 1.$ Therefore, for all $k\geq 2$, one has ${\RR}^k f_{\R}{}_{!!}\Db^{\TT}_{M_{\infty}} \simeq 0$ and it only remains to show that ${\RR}^1f_{\R}{}_{!!}\Db^{\TT}_{M_{\infty}} \simeq 0$, that is to say, to show that 
\begin{equation}\label{equ:epi}
f_{\R}{}_{!!}(j^{-1}\Db^{\ttt}_{\widehat{M}\times \PP}) \overset{\partial_t -1}\longrightarrow  f_{\R}{}_{!!}(j^{-1}\Db^{\ttt}_{\widehat{M}\times \PP})
\end{equation} 
is an epimorphism. Let us denote by $p_1 : \widehat{M}\times \PP \to \widehat{M}$ the first projection. Then, if $u$ and $v$ are two sections of $f_{\R}{}_{!!}(j^{-1}\Db^{\ttt}_{\widehat{M}\times \PP})$ such that $\partial_t v -v = u$, it is clear that $\supp(v) \subset p_1(\supp(u)) \times \PP.$ Using again Lemma $6.2.4$ of \cite{Dagn16}, this proves that~(\ref{equ:epi}) is an epimorphism.
\end{proof}

\begin{proposition}
Let $f : X_{\infty} = (X,\widehat{X}) \to Y_{\infty} = (Y,\widehat{Y})$ be a morphism of complex bordered spaces. The integration of distributions along the fibers of $f_{\R}$ induces a morphism of double complexes
\begin{equation}\label{equ:enhintegration}
\int_{f_{\R}} : f_{\R}{}_{!!} \Db^{\TT, \bullet+d_X, \bullet+d_X}_{X_{\infty}} \rightarrow \Db^{\TT, \bullet+d_Y, \bullet+d_Y}_{Y_{\infty}}
\end{equation}
and thus, a morphism
\begin{equation}\label{equ:derenhintegration}
\int_{f_{\R}} : {\EE} f_{!!} \Omega^{\EE,p+d_X}_{X_{\infty}} [d_X] \rightarrow \Omega^{\EE,p+d_Y}_{Y_{\infty}} [d_Y]
\end{equation}
\noindent in ${\EE}^{\bb}(\C^{\Sub}_{Y_{\infty}})$ for each $p\in \Z.$ 
\end{proposition}
\begin{proof}
Due to the specific form of $f_{\R} = f \times \text{id}_{\R}$, it is clear that~(\ref{equ:enhintegration}) is well-defined. Moreover, by the same proof as the one of Lemma~\ref{lem:acyclic}, one can show that $\Db^{\TT, p+d_X, q+d_X}_{X_{\infty}}$ is $f_{\R}{}_{!!}$-acyclic for all $(p,q)\in \Z^2.$ Hence the conclusion.
\end{proof}

\begin{proposition}\label{prop:keyprop1}
Let $f : X_{\infty}=(X,\widehat{X}) \to Y_{\infty}=(Y,\widehat{Y})$ be a morphism of complex bordered spaces and let $\mathcal{N} \in \DD^{\bb}_{\qgood}(\mathscr{D}_{Y_{\infty}}).$

\begin{enumerate}[label=(\roman*)]
\item (\cite{Kash16}, Proposition $4.15$ (i)) There is a natural isomorphism
\begin{equation}\label{equ:Functorial1}
\DR^{\EE}_{X_{\infty}}({\DD}f^* \mathcal{N})[d_X] \simeq {\EE}f^!\DR^{\EE}_{Y_{\infty}}(\mathcal{N})[d_Y].
\end{equation}

\item If $f$ extends to a holomorphic map $\hat{f} : \widehat{X}\to \widehat{Y},$ then applying (\ref{equ:Functorial1}) to $\mathcal{N}=\mathscr{D}_{Y_{\infty}}$ gives an isomorphism
\begin{equation}\label{equ:modadjointenh}
\Omega^{\EE}_{X_{\infty}} \overset{\LLL}\otimes_{\mathscr{D}_{X_{\infty}}} \mathscr{D}_{X_{\infty}\to Y_{\infty}} [d_X] \xrightarrow{\sim} {\EE}f^{!} \Omega^{\EE}_{Y_{\infty}} [d_Y].
\end{equation}

\item This morphism induces, thanks to $1_{X_{\infty}\to Y_{\infty}}$, a morphism 
\[
\Omega^{\EE}_{X_{\infty}} [d_X] \to {\EE}f^{!} \Omega^{\EE}_{Y_{\infty}} [d_Y]
\]
in $\EE^{\bb}(\C^{\Sub}_{X_{\infty}})$, which is equivalent to the adjoint of (\ref{equ:derenhintegration}) when $p=0.$
\end{enumerate}
\end{proposition}
\begin{proof}
The isomorphism~(\ref{equ:modadjointenh}) is built as an enhancement on bordered spaces of~(\ref{equ:modadjointtemp}). The conclusion follows from Proposition~\ref{prop:tempintegration2}.
\end{proof}

\begin{proposition}
Let $f : X_{\infty} = (X,\widehat{X}) \to Y_{\infty} = (Y,\widehat{Y})$ be a morphism of complex bordered spaces such that $f$ extends to a holomorphic submersion $\hat{f} : \widehat{X} \to \widehat{Y}.$ The pullback of distributions by $f_{\R}$ induces a morphism of double complexes
\begin{equation}\label{equ:enhpullback}
f_{\R}^* : f_{\R}^{-1} \Db^{\TT, \bullet, \bullet}_{Y_{\infty}} \rightarrow \Db^{\TT, \bullet, \bullet}_{X_{\infty}}
\end{equation}
and thus, a morphism
\begin{equation}\label{equ:derenhpullback}
f_{\R}^* : {\EE}f^{-1} \Omega^{\EE, p}_{Y_{\infty}} \rightarrow \Omega^{\EE,p}_{X_{\infty}}
\end{equation}
\noindent in ${\EE}^{\bb}(\C^{\Sub}_{X_{\infty}})$ for each $p\in \Z.$ 
\end{proposition}
\begin{proof}
The first morphism is again well-defined thanks to the specific form of $f_{\R}.$ The second one is obtained by the exactness of $f_{\R}^{-1}.$
\end{proof}

\begin{proposition}\label{prop:keyprop2}
Let $f : X_{\infty}=(X,\widehat{X}) \to Y_{\infty}=(Y,\widehat{Y})$ be a semi-proper morphism of complex bordered spaces and let $\mathcal{M} \in \DD^{\bb}_{\qgood}(\mathscr{D}_{X_{\infty}}).$ 
\begin{enumerate}[label=(\roman*)]
\item (\cite{Kash16}, Proposition $4.15$ (ii)) There is a natural isomorphism
\begin{equation} \label{equ:Functorial2}
\DR^{\EE}_{Y_{\infty}}({\DD}f_* \mathcal{M}) \simeq {\EE}f_*\DR^{\EE}_{X_{\infty}}(\mathcal{M}).
\end{equation}
\item (\cite{Kash01}, Lemma $7.4.10$) If $f$ extends to a holomorphic map $\hat{f} : \widehat{X}\to \widehat{Y}$, then (\ref{equ:Functorial2}) is induced by a morphism

\begin{equation}\label{equ:modenhpullback}
\mathscr{D}_{X_{\infty}\to Y_{\infty}} \overset{\LLL}\otimes_{f^{-1}\mathscr{D}_{Y_{\infty}}} {\EE}f^{-1}\Hol^{\EE}_{Y_{\infty}} \to \Hol^{\EE}_{X_{\infty}},
\end{equation}
which is an enhancement of~(\ref{equ:modpullback}). 
\item The morphism~(\ref{equ:modenhpullback}) induces, thanks to $1_{X_{\infty}\to Y_{\infty}}$, a morphism
\[
{\EE}f^{-1} \Hol^{\EE}_{Y_{\infty}} \rightarrow \Hol^{\EE}_{X_{\infty}}
\]
which is equivalent to~(\ref{equ:derenhpullback}) when $p=0,$ if $\hat{f}$ is a holomorphic submersion. 
\end{enumerate}
\end{proposition}
\noindent
Let us now introduce a third operation.

\begin{proposition}\label{prop:mult}
Let $X_{\infty} = (X,\widehat{X})$ be a complex bordered space and $\varphi : X \to \C$ a tempered function at infinity, i.e. $\varphi \in \Gamma(X, \CC_{\widehat{X}}^{\infty, \ttt}).$ Then, there is a morphism 

\begin{equation}\label{equ:basicmult}
{\mu_{-\Re \varphi}}_* \Db^{\TT,p,q}_{X_{\infty}} \to \Db^{\TT,p,q}_{X_{\infty}}
\end{equation}
defined by $\omega \mapsto e^{\varphi}\omega$ for any $(p,q) \in \Z.$ If moreover $\varphi$ is holomorphic, this gives rise to a morphism of complexes 

\begin{equation}\label{equ:basicmult2}
{\mu_{-\Re \varphi}}_* \Db^{\TT,p,\bullet}_{X_{\infty}} \to \Db^{\TT,p,\bullet}_{X_{\infty}}
\end{equation}
for each $p\in \Z.$ This morphism induces itself a morphism
\begin{equation}\label{equ:dermult}
\C_{\{t=-\Re\varphi(x)\}} \overset{+}\otimes\: \Omega^{\EE,p}_{X_{\infty}} \to \Omega^{\EE,p}_{X_{\infty}}
\end{equation}
in $\EE^{\bb}(\C^{\Sub}_{X_{\infty}})$ for each $p\in \Z.$
\end{proposition}

\begin{proof}
Let us consider $U \in \text{Op}^{\Sub,c}_{X_{\infty}\times \R_{\infty}}.$ Then, for any $(p,q) \in \Z^2$, we define the map 
\[
\Gamma(U, {\mu_{-\Re \varphi}}_* \Db^{\TT,p,q}_{X_{\infty}}) = \Gamma(\mu^{-1}_{-\Re \varphi}(U), \Db^{\TT,p,q}_{X_{\infty}}) \to \Gamma(U, \Db^{\TT,p,q}_{X_{\infty}})
\]
by $\omega(x,t) \mapsto \omega(x,t+\Re \varphi(x)).$ (This little abuse of notation corresponds to the pullback of $\omega$ by $\mu_{\Re \varphi}.$) Since $\varphi$ is tempered, this map is well-defined. Moreover, since $\omega$ is a solution of $\partial_t \omega =\omega$, one can write $\omega(x,t)=e^{t}\rho(x)$ for a unique distributional form $\rho$. Hence 
\[
\omega(x,t+\Re \varphi(x)) = e^{t+\Re \varphi(x)}\rho(x) = e^{\Re \varphi(x)}\omega(x,t).
\] 
To obtain (\ref{equ:basicmult}), it is now enough to compose this map with 
\[
\Gamma(U, \Db^{\TT,(p,q)}_{X_{\infty}}) \ni \omega \mapsto e^{i\Im \varphi} \omega \in \Gamma(U, \Db^{\TT,(p,q)}_{X_{\infty}}),
\] 
which is of course well-defined since $|e^{i \Im \varphi}| = 1.$ Then, (\ref{equ:basicmult2}) follows from the equality $\overline{\partial}(e^{\varphi}\omega) = e^{\varphi} \overline{\partial}\omega$ if $\varphi$ is holomorphic and (\ref{equ:dermult}) from the exactness of ${\mu_{-\Re \varphi}}_*$.
\end{proof}

\begin{proposition}\label{prop:keyprop3}
Let $X_{\infty}=(X,\widehat{X})$ be a complex bordered space and let $\LL \in \DD^{\bb}_{\hol}(\mathscr{D}_{X_{\infty}})$ and $\mathcal{M} \in \DD^{\bb}_{\qgood}(\mathscr{D}_{X_{\infty}}).$

\begin{enumerate}[label=(\roman*)]
\item (\cite{Kash16}, Proposition $4.15$ (iii))There is a natural isomorphism
\begin{equation} \label{equ:Functorial3}
\DR^{\EE}_{X_{\infty}}(\LL \overset{\DD}\otimes \mathcal{M}) \simeq {\RR}\mathscr{I}\!hom^+(\Sol^{\,\EE}_{X_{\infty}}(\LL)) , \DR^{\EE}_{X_{\infty}}(\mathcal{M})).
\end{equation}
\item Let $\varphi \in \Hol_{\widehat{X}}(* \widehat{X}\backslash X)$. Then, ~(\ref{equ:Functorial3}) applied to $\mathcal{M}=\mathscr{D}_{X_{\infty}}\otimes_{\Hol_{X_{\infty}}}\Omega^{\otimes-1}_{X_{\infty}}$ and $\LL = \mathscr{E}^{\varphi}_{X | \widehat{X}}$ gives an adjoint morphism 
\begin{equation}\label{equ:modmult}
\C_{\{t=-\Re \varphi(x)\}}\overset{+}\otimes \left(\mathscr{E}^{\varphi}_{X | \widehat{X}} \overset{\DD}\otimes \Hol^{\EE}_{X_{\infty}}\right) \to \Hol^{\EE}_{X_{\infty}},
\end{equation}
that induces, thanks to the canonical section $e^{\varphi}$ of $\mathscr{E}^{\varphi}_{X | \widehat{X}}$, a morphism 
\[
\C_{\{t=-\Re\varphi(x)\}} \overset{+}\otimes\: \Hol^{\EE}_{X_{\infty}} \to \Hol^{\EE}_{X_{\infty}},
\]
which is equivalent to~(\ref{equ:dermult}) when $p=0.$
\end{enumerate}
\end{proposition}
\begin{proof}
The explicit construction of~(\ref{equ:modmult}) is made in \cite{Kash16}, Theorem $4.5$ (f)-(1), by using Lemma $9.6.3$ and Proposition $9.6.5$ of \cite{Dagn16}. The reader shall immediately see that these results prove (ii).
\end{proof}

\section{A remark on the enhanced Laplace transform}
Let us recall some facts about the enhanced Fourier-Sato functors, introduced in \cite{Kash16}. Let us fix $\V$ a $n$-dimensional complex vector space and $\V^*$ its complex dual. We consider the bordered spaces $\V_{\infty} = (\V,\overline{\V})$ and $\V^*_{\infty} = (\V^*,\overline{\V}^*)$ where $\overline{\V}$ (resp. $\overline{\V}^*$) is the projective compactification of $\V$ (resp $\V^*$). Let us also note $\langle -,- \rangle : \V \times \V^* \to \C$ the duality bracket.

\begin{definition}
The Laplace kernels are defined by 
\begin{align*}
L_{\V} &= \C_{\{t = \Re \langle z,w \rangle\}} \in \EE^{\bb}(\C^{\Sub}_{\V_{\infty}\times \V^*_{\infty}}), \\
L^a_{\V} &= \C_{\{t = -\Re \langle z,w \rangle\}} \in \EE^{\bb}(\C^{\Sub}_{\V_{\infty}\times \V^*_{\infty}}).
\end{align*}

\end{definition}
\noindent
Let us consider the correspondence 
\[
\V_{\infty} \overset{p}\longleftarrow\V_{\infty} \times\V_{\infty}^* \overset{q}\longrightarrow \V_{\infty}^*
\]
\noindent
where $p$ and $q$ are the canonical projections.

\begin{definition}
The enhanced Fourier-Sato functors $${}^{\EE}\! \mathcal{F}_{\V}, {}^{\EE}\! \mathcal{F}^a_{\V} : \EE^{\bb}(\C^{\Sub}_{\V_{\infty}}) \to \EE^{\bb}(\C^{\Sub}_{\V^*_{\infty}})$$ are defined by 
\begin{align*}
{}^{\EE}\! \mathcal{F}_{\V}(F) & = {\EE}q_{!!}(L_{\V} \overset{+}\otimes {\EE}p^{-1} F),\\
{}^{\EE}\! \mathcal{F}^a_{\V}(F) & = {\EE}q_{!!}(L^a_{\V} \overset{+}\otimes {\EE}p^{-1} F).
\end{align*}
\end{definition}

\begin{remark}
In \cite{Kash16}, the authors mainly work with ${}^{\EE}\!\mathcal{F}_{\V}$. However, it will be more convenient for us to use ${}^{\EE}\!\mathcal{F}^a_{\V}$ instead.
\end{remark}

\begin{theorem}[\cite{Kash16}, Theorem 5.2.]\label{thm:equivalence}
The enhanced Fourier-Sato functor ${}^{\EE}\! \mathcal{F}^a_{\V}$ is an equivalence of categories whose inverse is given by ${}^{\EE}\! \mathcal{F}_{\V^*}[2n].$ In particular, one has an isomorphism 
\begin{equation}\label{equ:equivalence}
\RHom {}^{\EE}(F_1,F_2) \simeq \RHom {}^{\EE}({}^{\EE}\! \mathcal{F}^a_{\V}(F_1),{}^{\EE}\! \mathcal{F}^a_{\V}(F_2)),
\end{equation}
\noindent
functorial in $F_1,F_2 \in \EE^{\bb}(\C^{\Sub}_{\V_{\infty}}).$
\end{theorem}

We can now restate Theorem $6.3$ of \cite{Kash16} with an additional explicit information. 

\begin{theorem}\label{thm:main}
There is a morphism of complexes
\begin{equation}\label{equ:explicitlaplace}
{q}_{\R !!} ({\mu_{-\langle z,w \rangle}}_* p_{\R}^{-1}\Db^{\TT,n,\bullet+n}_{\V_{\infty}}) \to \Db^{\TT,0,\bullet}_{\V^*_{\infty}}
\end{equation}
encoding the usual positive Laplace transform of distributions (with an extra real parameter), i.e. $\omega \mapsto \int_{q_{\R}} e^{\langle z,w \rangle} p_{\R}^*\omega$. 

\bigskip
\noindent
This morphism induces an isomorphism 
\begin{equation}\label{equ:derivedlaplace}
{}^{\EE}\!\mathcal{F}^a_{\V}(\Omega^{\EE}_{\V_{\infty}})[n] \xrightarrow{\sim} \Hol^{\EE}_{\V^*_{\infty}}
\end{equation}
in $\EE^{\bb}(\C^{\Sub}_{\V^*_{\infty}})$.
\end{theorem}

\begin{proof}
On one hand, using morphisms (\ref{equ:enhintegration}), (\ref{equ:enhpullback}) and (\ref{equ:basicmult}), we can define a morphism of complexes
\begin{align*}
q_{\R !!} ({\mu_{-\langle z,w \rangle}}_* p_{\R}^{-1}\Db^{\TT,n,\bullet+n}_{\V_{\infty}})  & \to q_{\R !!} ({\mu_{-\langle z,w \rangle}}_* \Db^{\TT,n,\bullet+n}_{\V_{\infty}\times \V^*_{\infty}})\\
& \to q_{\R !!} (\Db^{\TT,n,\bullet+n}_{\V_{\infty}\times \V^*_{\infty}})\\
& \to \Db^{\TT,0,\bullet}_{ \V^*_{\infty}}
\end{align*}
which clearly encodes the usual positive Laplace transform of distributions. This induces a morphism ${}^{\EE}\!\mathcal{F}^a_{\V}(\Omega^{\EE}_{\V_{\infty}})[n] \rightarrow\Hol^{\EE}_{\V^*_{\infty}}$ in $\EE^{\bb}(\C^{\Sub}_{\V^*_{\infty}})$. 

\bigskip
\noindent
On the other hand, Theorem $6.3$ of \cite{Kash16} states that there is a canonical isomorphism ${}^{\EE}\!\mathcal{F}^a_{\V}(\Omega^{\EE}_{\V_{\infty}})[n] \xrightarrow{\sim} \Hol^{\EE}_{\V^*_{\infty}}$. Looking at the proof of this theorem, the reader should see that this isomorphism is built by using isomorphisms (\ref{equ:modadjointenh}),(\ref{equ:Functorial2}) and (\ref{equ:Functorial3}). Hence, Propositions~\ref{prop:keyprop1}, \ref{prop:keyprop2} and \ref{prop:keyprop3} allow to conclude.  
\end{proof}

\section{Link with the Legendre transform}

In this section, we recall some definitions and propositions of the sections $5.4$, $6.2$ and $6.3$ of \cite{Kash16} and make use of our previous remark about the usual positive Laplace transform.

\begin{definition}
Let $M$ be a real analytic manifold and $U$ be a subanalytic open subset of $M$. A function $f : U \to \R$ is subanalytic on $M$ if its graph $\Gamma_{f} \subset U \times \R$ is subanalytic in $M \times \overline{\R}$. A continuous function $f : U \to \R$ is almost $\CC^{\infty}$-subanalytic on $M$ if there is a subanalytic $\CC^{\infty}$-function $g : U \to \R$ such that
\[
\exists\, C >0, \forall x\in U  : |f(x)-g(x)| < C.
\] 
In this case, we say that $g$ is in the (ASA)-class of $f$.
\end{definition}

In \cite{Kash16}, M. Kashiwara and P. Schapira make the conjecture that any continuous subanalytic function is almost $\CC^{\infty}$-subanalytic. 

\begin{definition}
Let $f : U \to \R$ be a continuous almost $\CC^{\infty}$-subanalytic function on $M$. For any open subanalytic set $V \subset M$ and any $r\in \Z$, we set 
\[
e^{-f}\Db^{\ttt,r}_M(V) = \{\omega \in \Db^r_M(U \cap V) : e^g \omega \in  \Db^{\ttt,r}_M(U \cap V)\},
\]
where $g$ is in the (ASA)-class of $f$. This definition does not depend on $g$ and the correspondence $V\in \text{Op}^{\Sub,c}_M \mapsto e^{-f}\Db^{\ttt,r}_M(V)$ clearly defines a quasi-injective subanalytic sheaf on $M$. 
\end{definition}

\begin{proposition}[\cite{Kash16}, Theorem $6.12.$ See also \cite{Dagn14}, Proposition $7.3.$]\label{prop:conv1}
Let $M$ be a real analytic manifold and $U$ be a subanalytic open subset of $M$. Let $f : U \to \R$ be a continuous almost $\CC^{\infty}$-subanalytic function on $M$. There is an isomorphism 
\[
e^{-f} \Db^{\ttt,r}_M \simeq {\RR}\mathscr{I}\!hom^{\EE}(\C_{\{t \geq f(x), x\in U\}}, \Db^{\EE,r}_{M})
\] for each $r\in \Z,$ which is given on sections by $\omega \mapsto e^t \omega.$ In particular, the right hand side is concentrated in degree $0$.  
\end{proposition}
\noindent
One can notice an immediate corollary :

\begin{corollary}\label{cor:conv2}
Let $M$ be a real analytic manifold and let $f : U \to \R$ be a continuous almost $\CC^{\infty}$-subanalytic function on $M$. Let $S$ be a subanalytic closed subset of $U$, then
\[
\mathscr{I}\Gamma_{S}(e^{-f}\Db^{\ttt,r}_M) \simeq {\RR}\mathscr{I}\!hom^{\EE}(\C_{\{t \geq f(x), x \in S\}}, \Db^{\EE,r}_{M})
\] 
for each $r\in \Z,$ which is given on sections by $\omega \mapsto e^t \omega.$ In particular, the right hand side is concentrated in degree $0$.
\end{corollary}
\noindent
Thanks to Proposition~\ref{prop:conv1}, one can introduce the following definition :

\begin{definition}
Let $U$ be an open subset of a complex manifold $X$ and $f : U \to \R$ be a continuous almost $\CC^{\infty}$-subanalytic function on $X$. For each $p\in \Z,$ one defines the complex of subanalytic sheaves $e^{-f} \Omega^{\ttt,p}_{X}$ as the Dolbeault complex $$0 \to e^{-f}\Db^{\ttt,p,0}_X \overset{\bar{\partial}}\to e^{-f}\Db^{\ttt,p,1}_X \to \dots \to e^{-f}\Db^{\ttt,p,d_X}_X \to 0.$$
\end{definition}
\noindent
Let us now focus on an important application.

\begin{definition}
Let $\V$ be a complex vector space of dimension $n$ and $f : \V \to \R \cup \{+\infty\}$ a function.

\begin{enumerate}[label=(\roman*)]

\item One says that $f$ is a closed proper convex function on $\V$ if its epigraph $$\{(z,t) \in \V \times \R : t \geq f(z)\}$$ is closed, convex and non-empty.
\item One notes $\text{Conv}(\V)$ the set of closed proper convex functions on $\V$.
\item For any $f\in \text{Conv}(\V)$, one sets $\dom(f) = f^{-1}(\R)$ and call it the domain of $f$. This set is convex and non-empty. The interior of this domain will be noted $\dom^{\circ}(f).$
\item For any $f \in \text{Conv}(\V)$, one defines a function $f^* : \V^* \to \R\cup \{+\infty\}$ by setting $$f^*(w) = \sup_{z\in \dom(f)} (\Re \langle z,w \rangle -f(z)).$$ It is called the Legendre transform of $f$. It is an element of $\text{Conv}(\V^*).$
\item For any $f \in \Conv(\V)$, one denotes by $H(f)$ the real affine space generated by $\dom(f)$ and one sets $E(f) = H(f^*)^{\bot}.$ One also sets $$d(f) = \dim_{\R} E(f) = \codim_{\R} H(f^*).$$
\end{enumerate}
\end{definition}

\begin{lemma}[\cite{Kash16}, Theorem. 5.9]\label{lem:legendre}
Let $f\in \Conv(\V)$. One has an isomorphism 
\begin{equation}\label{equ:Satolegendre}
{}^{\EE}\!\mathcal{F}^a_{\V}(\C_{\{t \geq f(z)\}}) \simeq \C_{\{t \geq -f^*(w), w \in \dom^{\circ}(f^*)\}} \otimes \, \text{or}_{E(f)} [-d(f)].
\end{equation}
\end{lemma}

Let $f : \V \to \R$ be a continuous almost $\CC^{\infty}$-subanalytic function on $\overline{\V}$ and let $S$ be a subanalytic closed subset of $\V$. Let us denote by $f_S$ the function which is equal to $f$ on $S$ and to $+\infty$ on $\V \backslash S.$ Assume that 

\begin{enumerate}[label=(\roman*)]
\item the function $f_S$ is convex,
\item $E(f_S) = \{0\},$
\item the convex set $\dom^{\circ}(f_S^*)$ is subanalytic,
\item the function $f_S^* : \dom^{\circ}(f_S^*) \to \R$ is continuous and almost $\CC^{\infty}$-subanalytic on $\overline{\V}^*$.
\end{enumerate}

Then, by using successively Corollary~\ref{cor:conv2}, the isomorphisms~(\ref{equ:equivalence}) and (\ref{equ:derivedlaplace}), Lemma~\ref{lem:legendre} and finally Proposition~\ref{prop:conv1}, one gets an isomorphism

\begin{equation}\label{equ:isolegendre}
H^n_{S}(\V, e^{-f}\Omega^{\ttt}_{\overline{\V}}) \xrightarrow{\sim} H^0(\V^*, e^{f_S^*}\Hol^{\ttt}_{\overline{\V}^*}) \simeq e^{f_S^*}\Db^{\ttt}_{\overline{\V}^*}(\dom^{\circ}(f_S^*))\cap \Hol_{\overline{\V}^*}(\dom^{\circ}(f_S^*)).
\end{equation}

This is simply Corollary $6.15$ of \cite{Kash16} with an additional closed support. We can make this isomorphism more explicit.

\begin{proposition}\label{prop:mainlegendre}
There is a commutative diagram 

\[
\xymatrix{H^n_{S}(\V, e^{-f}\Omega^{\ttt}_{\overline{\V}}) \ar[r]^{\sim} & H^0(\V^*, e^{f_S^*}\Hol^{\ttt}_{\overline{\V}^*}) \ar[d] \\
\Gamma(\V, \mathscr{I}\Gamma_{S} (e^{-f}\Db^{\ttt, n,n}_{\overline{\V}}))\ar[r]\ar[u] & \Gamma(\V^*, e^{f_S^*}\Db^{\ttt}_{\overline{\V}^*})}
\]
where the left arrow is defined by the Dolbeault resolution of $e^{-f}\Omega^{\ttt}_{\overline{\V}}$, the right arrow by the inclusion and the bottom arrow by $\omega \mapsto \LL^+ \omega := \int_q e^{\langle z,w \rangle} p^{*}\omega.$ In particular, the isomorphism (\ref{equ:isolegendre}) can be explicitly computed by

\begin{equation}
\frac{\Gamma(\V, \mathscr{I}\Gamma_{S} (e^{-f}\Db^{\ttt, n,n}_{\overline{\V}}))}{\bar{\partial}\Gamma(\V, \mathscr{I}\Gamma_{S} (e^{-f}\Db^{\ttt, n,n-1}_{\overline{\V}}))} \ni [\omega] \mapsto \LL^+ \omega \in H^0(\V^*, e^{f_S^*}\Hol^{\ttt}_{\overline{\V}^*}).
\end{equation} 
\end{proposition}

\begin{proof}
By construction and using (\ref{equ:explicitlaplace}), the map
\begin{align*}
\Gamma(\V, \mathscr{I}\Gamma_{S} (e^{-f}\Db^{\ttt, n,n}_{\overline{\V}})) &\xrightarrow{} H^n_{S}(\V, e^{-f}\Omega^{\ttt}_{\overline{\V}})\\ &\xrightarrow{\sim} H^0(\V^*, e^{f_S^*}\Hol^{\ttt}_{\overline{\V}^*}) \\ &\xrightarrow{} \Gamma(\V^*, e^{f_S^*}\Db^{\ttt}_{\overline{\V}^*})
\end{align*}
is given by $\omega \mapsto e^{-t} \int_{q_{\R}} e^{\langle z,w \rangle} p_{\R}^{*}(e^t \omega) = \LL^+ \omega.$ Then, the conclusion follows from the quasi-injectivity of $ e^{-f}\Db^{\ttt,p,q}_{\overline{\V}}$ for all $(p,q)\in \Z^2.$
\end{proof}

\section{Application I : Polya's theorem}

Let $\V$ be a one-dimensional complex vector space. Let us denote by $\PPP$ (resp. $\PPP^*$) the projective compactification of $\V$ (resp. $\V^*$) and  recall that $\Hol^{\ttt}_{\PPP}$ (resp. $\Omega^{\ttt}_{\PPP}$) is concentrated in degree $0$ and is a subanalytic subsheaf of $\Hol_{\PPP}$ (resp. $\Omega_{\PPP}$). If $U\in \Op^{\Sub,c}_{\PPP}$, one simply has $\Hol^t_{\PPP}(U) = \Hol_{\PPP}(U) \cap \Db^{\ttt}_{\PPP}(U)$ (resp. $\Omega^{\ttt}_{\PPP}(U) = \Omega_{\PPP}(U) \cap \Db^{\ttt,1,0}_{\PPP}(U)$). We shall also use the sheaf $\Omega^{\ttt_{\infty}}_{\PPP}$ of holomorphic forms tempered only at infinity. 

\medskip

Let $K$ be a non-empty convex compact subset of $\V$ and let $$h_K : w\in \V^* \mapsto \sup_{z\in K} \Re \langle z,w \rangle$$ be its support function. Let us choose a hermitian norm $||\cdot||$ on $\V$ and denote also by $||\cdot||$ the dual norm on $\V^*.$ The classical Polya's theorem (see e.g.\cite{Bere95}, sections $1.3$ and $1.4$) states that there is a (topological) isomorphism between 
\[
\Omega^{0}(\V \backslash K) := \{\omega\in \Omega_{\V}(\V \backslash K) : \lim_{z\to \infty} \omega(z) = 0\}
\] 
and 
\[\text{Exp}(K) := \{v \in \Hol_{\V^*}(\V^*) : \forall \varepsilon>0,\: \sup_{w\in \V^*} |v(w)|e^{-h_K(w)-\varepsilon ||w||} < \infty \}.
\] These spaces do not depend on the chosen norm. 

\medskip
\noindent
Given a global $\C$-linear coordinate $z$ of $\V$ and $w$ its dual coordinate, this isomorphism can be explicited by $\Omega^{0}(\V \backslash K) \ni u(z)dz \mapsto v \in \text{Exp}(K)$, with
\[
v(w) = \int_{C(0,r)^+} e^{zw} u(z)dz,
\] 
where $C(0,r)^+$ is a positively oriented circle of center $0$ and radius $r>0$, which encloses $K$. Of course, the integral does not depend on the chosen circle.

\bigskip

Let us fix a non-empty convex compact subset $K$ of $\V$. For all $\varepsilon>0$, we consider the thickening of $K$ by $\varepsilon$, that is to say $K_{\varepsilon} := K+\overline{D}(0,\varepsilon),$ where $\overline{D}(0,\varepsilon)=\{z\in \V : ||z||\leq \varepsilon\}.$ Let us consider the null function $f=0$ on $\V$. For all $\varepsilon>0$, we thus get a function $f_{K_{\varepsilon}}$ defined by
\[
f_{K_{\varepsilon}}(z) = \begin{cases} 0 \,\, &\text{if}\,\, z \in K_{\varepsilon}, \\ +\infty \,\, &\text{else}. \end{cases}
\] 
Clearly, this function is convex of domain $K_{\varepsilon}$. Moreover, its Legendre transform is given by 
\[
f_{K_{\varepsilon}}^*(w) = \sup_{z\in K_{\varepsilon}} \Re\langle z,w \rangle = h_{K_{\varepsilon}}(w) = h_{K}(w)+ h_{\overline{D}(0,\varepsilon)}(w) = h_{K}(w)+\varepsilon ||w||,
\] 
for all $w\in \V^*.$ In particular $\dom^{\circ}(f_{\varepsilon}^*) = \V^*.$ In order to apply Proposition~\ref{prop:mainlegendre}, we will assume throughout this section that $K_{\varepsilon}$ is subanalytic and that $h_{K_{\varepsilon}}$ is almost $\CC^{\infty}$-subanalytic on $\PPP^*$ for all $\varepsilon$. Thus, for all $\varepsilon,$ we get an isomorphism 
\begin{equation}\label{equ:temppolya1}
H^1_{K_{\varepsilon}}(\V, \Omega^{\ttt}_{\PPP}) \xrightarrow{\sim} e^{h_{K_{\varepsilon}}}\Hol^{\ttt}_{\PPP^*}(\V^*)
\end{equation}
\noindent
given by the positive Laplace transform. We shall show that the projective limit on $\varepsilon\to 0$ of this isomorphism is equivalent to Polya's theorem.

\begin{proposition}\label{prop:excision}
Let $\varepsilon>0.$ One has a canonical isomorphism 
\begin{equation}\label{equ:excision}
\Omega^{\ttt}_{\PPP}(\V \backslash K_{\varepsilon})/\Omega^{\ttt}_{\PPP}(\V) \xrightarrow{\sim} H^1_{K_{\varepsilon}}(\V, \Omega^{\ttt}_{\PPP})
\end{equation}
given by 
\[
\Omega^{\ttt}_{\PPP}(\V \backslash K_{\varepsilon})/\Omega^{\ttt}_{\PPP}(\V) \ni [\omega] \mapsto [\bar{\partial} \underline{\omega}] \in \frac{\Gamma(\V,\mathscr{I}\Gamma_{K_{\varepsilon}}(\Db^{\ttt,1,1}_{\PPP}))}{\bar{\partial}\Gamma(\V,\mathscr{I}\Gamma_{K_{\varepsilon}}(\Db^{\ttt,1,0}_{\PPP}))},
\]
where $\underline{\omega}$ is a distributional extension of $\omega$ to $\V.$
\end{proposition}
\begin{proof}
1) Consider the excision distinguished triangle 

\begin{equation}\label{equ:disttriangle}
{\RR}\Gamma(\V, \mathscr{I}\Gamma_{K_{\varepsilon}}(\Omega^{\ttt}_{\PPP})) \to {\RR}\Gamma(\V, \Omega^{\ttt}_{\PPP}) \to {\RR}\Gamma(\V\backslash K_{\varepsilon}, \Omega^{\ttt}_{\PPP}) \overset{+1}\to.
\end{equation}
This gives the following exact sequence : 

\begin{center}
\begin{tikzcd}
0 \rar & H^0_{K_{\varepsilon}}(\V, \Omega^{\ttt}_{\PPP}) \rar & H^0(\V, \Omega^{\ttt}_{\PPP}) \rar
             \ar[draw=none]{d}[name=X, anchor=center]{}
    & H^0(\V\backslash K_{\varepsilon}, \Omega^{\ttt}_{\PPP}) \ar[rounded corners,
            to path={ -- ([xshift=2ex]\tikztostart.east)
                      |- (X.center) \tikztonodes
                      -| ([xshift=-2ex]\tikztotarget.west)
                      -- (\tikztotarget)}]{dll}[at end]{} \\      
  & H^1_{K_{\varepsilon}}(\V, \Omega^{\ttt}_{\PPP}) \rar & H^1(\V, \Omega^{\ttt}_{\PPP}) \rar & H^1(\V\backslash K_{\varepsilon}, \Omega^{\ttt}_{\PPP}) \rar & \cdots
\end{tikzcd}
\end{center} 

\medskip
\noindent
Firstly, it is clear that $H^0_{K_{\varepsilon}}(\V, \Omega^{\ttt}_{\PPP}) \simeq 0$ since a non-trivial holomorphic form can't be supported by a compact subset. Secondly, the surjectivity of $\overline{\partial} : \Db^{\ttt,1,0}_{\PPP}(\V) \to \Db^{\ttt,1,1}_{\PPP}(\V)$ (see \cite{Horm58} and \cite{Loja69}) implies that $H^1(\V, \Omega^{\ttt}_{\PPP}) \simeq 0$.  Hence we get the exact sequence 
\[
0 \to \Omega^{\ttt}_{\PPP}(\V) \to \Omega^{\ttt}_{\PPP}(\V \backslash K_{\varepsilon}) \to  H^1_{K_{\varepsilon}}(\V, \Omega^{\ttt}_{\PPP}) \to 0
\] 
which proves the first statement.

\bigskip
\noindent
2) If follows from~(\ref{equ:disttriangle}) that ${\RR}\Gamma(\V, \mathscr{I}\Gamma_{K_{\varepsilon}}\Omega^{\ttt}_{\PPP})$ is canonically isomorphic to the mapping cone $M(\rho_{K_{\varepsilon}})$ of the restriction morphism 
\[
\rho_{K_{\varepsilon}} : \Db^{\ttt, 1, \bullet}_{\PPP}(\V) \to \Db^{\ttt, 1, \bullet}_{\PPP}(\V\backslash K_{\varepsilon})
\] 
shifted by $-1$. We know that $M(\rho_{K_{\varepsilon}})[-1]$ is a complex concentrated in degrees $0,1$ and $2$ of the form 
\[
\Db^{\ttt,1,0}_{\PPP}(\V) \to \Db^{\ttt,1,1}_{\PPP}(\V) \oplus \Db^{\ttt,1,0}_{\PPP}(\V\backslash K_{\varepsilon}) \to \Db^{\ttt, 1, 1}_{\PPP}(\V\backslash K_{\varepsilon}),
\] 
where the differentials in degrees $0$ and $1$ are given by the matrices 
\[
\binom{\bar{\partial}}{-\rho_{K_{\varepsilon}}} \,\,\,\,\, \text{and} \,\,\,\,\, \begin{pmatrix}-\rho_{K_{\varepsilon}} & -\bar{\partial}\end{pmatrix}.
\] 
We have to show that 
\[
\binom{\bar{\partial} \underline{\omega}}{0} \,\,\,\,\, \text{and} \,\,\,\,\, \binom{0}{\omega}
\] 
are two $1$-cycles of this complex which are in the same cohomology class. This is clear since 
\[
\binom{\bar{\partial}}{-\rho_{K_{\varepsilon}}} \underline{\omega} + \binom{0}{\omega} = \binom{\bar{\partial} \underline{\omega}}{0}.
\]
Hence the conclusion.
\end{proof}

\begin{corollary}\label{cor:compactprojlim}
One has a canonical isomorphism 
\begin{equation}\label{equ:temppolya2}
\Omega^{\ttt_{\infty}}_{\PPP}(\V\backslash K)/\Omega^{\ttt}_{\PPP}(\V) \xrightarrow{\sim} \varprojlim_{\varepsilon\to 0} H^1_{K_{\varepsilon}}(\V, \Omega^{\ttt}_{\PPP}).
\end{equation}
\noindent
Let $\varepsilon>0$ and let $\psi_{\varepsilon}$ be a $\CC^{\infty}$-cutoff function which is equal to $1$ on $\V \backslash K_{\varepsilon}$ and to $0$ on $K_{\varepsilon/2}$. Let $\omega \in \Omega^{\ttt_{\infty}}_{\PPP}(\V\backslash K).$ Then the image of $[\omega]$ through the canonical map
\[
\Omega^{\ttt_{\infty}}_{\PPP}(\V\backslash K)/\Omega^{\ttt}_{\PPP}(\V) \rightarrow H^1_{K_{\varepsilon}}(\V, \Omega^{\ttt}_{\PPP})
\]
is given by $[\bar{\partial} (\psi_{\varepsilon} \omega)].$
\end{corollary}

\begin{proof}
Simply notice that there are inclusions
\[
\Omega^{\ttt_{\infty}}_{\PPP}(\V\backslash K_{\varepsilon}) \subset \Omega^{\ttt}_{\PPP}(\V\backslash K_{2\varepsilon}) \subset \Omega^{\ttt_{\infty}}_{\PPP}(\V\backslash K_{3\varepsilon})
\]
for all $\varepsilon>0$ and that 
\[
\varprojlim_{\varepsilon\to 0} \Omega^{\ttt_{\infty}}_{\PPP}(\V\backslash K_{\varepsilon}) \simeq \Omega^{\ttt_{\infty}}_{\PPP}(\V\backslash K).
\]
\end{proof}

\begin{remark}
Note that in 
\[
\Omega^{\ttt_{\infty}}_{\PPP}(\V\backslash K) = \{\omega \in \Omega_{\V}(\V \backslash K) : \omega \,\, \text{is tempered at}\,\, \infty\},
\]
one can replace the condition "$\omega$ is tempered at infinity" by the condition "$\omega$ has polynomial growth at infinity". Indeed, thanks to Cauchy's inequalities, the polynomial growth of $\omega$ implies the polynomial growth of all its derivatives. 
\end{remark}

\begin{definition}
We set 
\[
\text{Exp}^{\ttt}(K) = \varprojlim_{\varepsilon \to 0} e^{h_{K_{\varepsilon}}}\Hol^{\ttt}_{\PPP^*}(\V^*) \simeq  \{v \in \Hol_{\V^*}(\V^*) : \forall \varepsilon>0, v \in e^{h_{K_{\varepsilon}}}\Db^{\ttt}_{\PPP^*}(\V^*)\}.
\]
\end{definition}

\noindent
This set does not depend on the chosen norm.
\begin{theorem}\label{thm:mainpolya}
There is a canonical isomorphism of $\C$-vector spaces
\begin{equation}
\Omega^{\ttt_{\infty}}_{\PPP}(\V\backslash K)/\Omega^{\ttt}_{\PPP}(\V) \xrightarrow{\sim} \Exp^{\ttt}(K).
\end{equation}
\noindent
 Given a global $\C$-linear coordinate $z$ of $\V$ and $w$ its dual coordinate, this isomorphism can be explicited by $[u(z)dz] \mapsto v$ with
\[
v(w) = \int_{C(0,r)^+} e^{zw} u(z)dz,
\] 
where $C(0,r)^+$ is a positively oriented circle, which encloses $K$.
\end{theorem}
\begin{proof}
We apply $\underset{\varepsilon\to 0}\varprojlim$ to (\ref{equ:temppolya1}) as well as (\ref{equ:temppolya2}) to get the isomorphisms
\[
\Omega^{\ttt_{\infty}}_{\PPP}(\V\backslash K)/\Omega^{\ttt}_{\PPP}(\V) \xrightarrow{\sim}\varprojlim_{\varepsilon\to 0} H^1_{K_{\varepsilon}}(\V, \Omega^{\ttt}_{\PPP})  \xrightarrow{\sim }\text{Exp}^{\ttt}(K).
\]
Let us explicit the composition of these two maps within coordinates. Let $u(z)dz$ be in $\Omega^{\ttt_{\infty}}_{\PPP}(\V\backslash K)$ and let us fix $r>0$ such that $K \subset D(0,r).$ Let us consider $\varepsilon>0$ small enough such that $K \subsetneq K_{\varepsilon} \subsetneq D(0,r).$ Let us also choose a cutoff function $\psi_{\varepsilon}$ as in Corollary~\ref{cor:compactprojlim}. Then, applying this corollary, we see that the image of $[u(z)dz]$ in $ e^{h_{K_{\varepsilon}}}\Hol^{\ttt}_{\PPP}(\V^*)$ is given by $v$, where
\begin{align*}
v(w) = \LL^+_w( \bar{\partial} (\psi_{\varepsilon} u(z)dz)) &= \int_{\V} e^{zw} \bar{\partial} (\psi_{\varepsilon} u(z)dz) \, = \int_{\V} \bar{\partial} (e^{zw}\psi_{\varepsilon} u(z)dz) \\
& \underset{(1)}= \int_{\overline{D}(0,r)} \bar{\partial} (e^{zw}\psi_{\varepsilon} u(z)dz) \underset{(2)}= \int_{C(0,r)^+}e^{zw}\psi_{\varepsilon}u(z)dz \\
& \underset{(3)}= \int_{C(0,r)^+}e^{zw}u(z)dz,
\end{align*}
where $(1)$ comes from the holomorphicity of $e^{zw}\psi_{\varepsilon}u(z)dz$ on the open set $$\V \backslash K_{\varepsilon} \supset \V \backslash \overline{D}(0,r),$$ $(2)$ from Green's theorem and $(3)$ from the fact that $\psi_{\varepsilon}=1$ on $C(0,r) \subset \V \backslash K_{\varepsilon}.$

\bigskip
\noindent
To conclude, we remark that this formula remains unchanged for smaller $\varepsilon>0$. Hence, it is the image of $[u(z)dz]$ in $\text{Exp}^{\ttt}(K).$
\end{proof}

\begin{remark}\label{rem:polya}
Theorem~\ref{thm:mainpolya} is actually nothing more but Polya's theorem. Firstly, the canonical map
\[
\Omega^{0}(\V \backslash K) \ni \omega \mapsto [\omega] \in \Omega^{\ttt_{\infty}}_{\PPP}(\V\backslash K)/\Omega^{\ttt}_{\PPP}(\V),
\] 
is clearly injective. Secondly, the inclusion $\text{Exp}(S) \subset \text{Exp}^{\ttt}(S)$ is an equality. Indeed, if $e^{-h_{K_{\varepsilon}}}v$ is tempered at infinity, then $e^{-h_{K_{2\varepsilon}}}v$ is bounded.
\end{remark}

\section{Application II : Méril's theorem}

We keep the same conventions that in the previous section. Méril's theorem (see \cite{Meril83}) is a kind of non-compact analogue of Polya's theorem. Let $S$ be a non-empty closed convex non-compact subset of $\V$ which contains no lines. Let us set  
\[
S_{\infty} = \{z\in \V : z+S \subset S \}
\] 
the asymptotic cone of $S$ and 
\[
S_{\infty}^*= \{w \in \V^* : \forall z\in S_{\infty}, \: \Re \langle z,w \rangle \leq 0 \}
\] 
the polar cone of $S_{\infty}$. It is clear that $S_{\infty}^*$ is a closed convex proper cone of $\V^*$ with non empty interior, since $S$ does not contain any line. (We refer to \cite{Ausl03} for more details on convex geometry and asymptotic cones.) Let $\xi_0 \in \V^*$ be a fixed point on the bisector of $S_{\infty}^*.$ For all $\varepsilon'>0$, let us set
\[
\mathscr{H}_S(\V, \varepsilon') := \frac{\{\omega\in \Omega_{\V}(\V \backslash S) : \forall r>\varepsilon> 0, \, \sup_{z \in S_r \backslash S^{\circ}_{\varepsilon}} ||e^{\langle z,\varepsilon'\xi_0 \rangle}\omega(z)|| < \infty)\}}{\{\omega\in \Omega_{\V}(\V) : \forall r>0, \, \sup_{z \in S_r} ||e^{\langle z,\varepsilon'\xi_0 \rangle}\omega(z)|| < \infty)\}}.
\] 
Set also
\[
\text{Exp}(S) := \{v\in \Hol_{\V^*}((S_{\infty}^*)^{\circ}) : \forall \varepsilon, \varepsilon'>0,\!\! \sup_{w \in S_{\infty}^*+\varepsilon'\xi_0} |v(w)|e^{-h_{S}(w)-\varepsilon ||w||} < \infty\}.
\]
These spaces do not depend on the chosen norm. Méril's theorem states that there is a (topological) isomorphism 
\[
\varprojlim_{\varepsilon'\to 0} \mathscr{H}_S(\V, \varepsilon') \xrightarrow{\sim} \text{Exp}(S).
\]
Given a global $\C$-linear coordinate $z$ of $\V$ and $w$ its dual coordinate, this isomorphism can be explicited by
\[
\varprojlim_{\varepsilon'\to 0}\mathscr{H}_S(\V,\varepsilon') \ni ([u_{\varepsilon'}(z)dz])_{\varepsilon'} \mapsto v \in \text{Exp}(S),
\]
with
\[
v(w) = \int_{\partial S_{\varepsilon}^+} e^{zw} u_{\varepsilon'}(z)dz,
\] 
where $\partial S_{\varepsilon}^+$ is the positively oriented boundary of any thickening $S_{\varepsilon}$. (Recall that the boundary of a plane convex set is always a rectifiable curve.) This integral does not depend on $\varepsilon, \varepsilon'$. 

\medskip
These functional spaces are deeply linked to analytic functionals with non-compact carrier and, according to \cite{Roev78}, are of interest in quantum field mechanics. 

\bigskip
Let us fix $S$ a non-empty closed convex non-compact subset of $\V$ which contains no lines and $\xi_0$ a point on the bisector of $S_{\infty}^*.$ For all $\varepsilon'>0,$ we consider the function $f_{\varepsilon'} : \V \to \R$ defined by $f_{\varepsilon'}(z)= \Re \langle z, \varepsilon'\xi_0 \rangle.$ For all $\varepsilon,\varepsilon'>0$, we thus get a function $f_{\varepsilon,\varepsilon'}:={(f_{\varepsilon'})}_{S_{\varepsilon}}$ defined by 
\[
f_{\varepsilon,\varepsilon'}(z)=\begin{cases} \Re \langle z, \varepsilon'\xi_0 \rangle \,\, &\text{if}\,\, z \in S_{\varepsilon}, \\ +\infty \,\, &\text{else}. \end{cases}
\]  
Clearly, this function is convex of domain $S_{\varepsilon}$. Moreover, its Legendre transform is given by 
\[
f_{\varepsilon,\varepsilon'}^*(w) = \sup_{z\in S_{\varepsilon}} \Re \langle z,(w-\varepsilon'\xi_0) \rangle = h_{S_{\varepsilon}}(w-\varepsilon'\xi_0),
\] 
for all $w\in \V^*.$ Since it is well-known that $\dom^{\circ}(h_S) = (S_{\infty}^*)^{\circ}$, one immediately gets that $\dom^{\circ}(f_{\varepsilon,\varepsilon'}^*) = (S_{\infty}^*)^{\circ}+\varepsilon'\xi_0.$ In particular, since this open cone is not empty, its generated affine space is $\V^*$. In order to apply Proposition~\ref{prop:mainlegendre}, we will assume throughout this section that $S_{\varepsilon}$ is subanalytic and that $h_{S_{\varepsilon}}$ is almost $\CC^{\infty}$-subanalytic on $\PPP^*$ for all $\varepsilon>0$. Hence we get an isomorphism 
\begin{equation}\label{equ:tempmeril1}
H^1_{S_{\varepsilon}}(\V, e^{-\langle z, \varepsilon'\xi_0 \rangle}\Omega^{\ttt}_{\PPP}) \xrightarrow{\sim} e^{h_{S_{\varepsilon}}(w-\varepsilon'\xi_0)}\Hol^{\ttt}_{\PPP^*}((S_{\infty}^*)^{\circ}+\varepsilon'\xi_0)
\end{equation}

\medskip
\noindent
given by the positive Laplace transform for all $\varepsilon, \varepsilon'>0.$ (Here $e^{-\langle z, \varepsilon'\xi_0 \rangle}\Omega^{\ttt}_{\PPP}$ is defined in the obvious way and is of course equal to $e^{-\Re\langle z, \varepsilon'\xi_0 \rangle}\Omega^{\ttt}_{\PPP}$.) We shall show that the projective limit on $\varepsilon, \varepsilon'\to 0$ of this isomorphism is equivalent to Méril's theorem. 

\bigskip
\noindent
One can easily adapt Proposition~\ref{prop:excision} and Corollary~\ref{cor:compactprojlim} to obtain 

\begin{proposition}\label{prop:tempmeril3}
For all $\varepsilon,\varepsilon'>0$ there is a canonical isomorphism 
\begin{equation}\label{equ:tempmeril2}
e^{-\langle z, \varepsilon'\xi_0 \rangle}\Omega^{\ttt}_{\PPP}(\V\backslash S_{\varepsilon})/e^{-\langle z, \varepsilon'\xi_0 \rangle}\Omega^{\ttt}_{\PPP}(\V) \xrightarrow{\sim} H^1_{S_{\varepsilon}}(\V, e^{-\langle z, \varepsilon'\xi_0 \rangle}\Omega_{\PPP}^{\ttt}).
\end{equation}
Let $\varepsilon, \varepsilon'>0$ and let $\psi_{\varepsilon}$ be a $\CC^{\infty}$-cutoff function which is equal to $1$ on $\V \backslash S_{\varepsilon}$ and to $0$ on $S_{\varepsilon/2}$. Let $\omega \in e^{-\langle z, \varepsilon'\xi_0 \rangle}\Omega^{\ttt_{\infty}}_{\PPP}(\V\backslash S)$ Then the image of $[\omega]$ through the canonical map 
\begin{align*}
e^{-\langle z, \varepsilon'\xi_0 \rangle}\Omega^{\ttt_{\infty}}_{\PPP}(\V\backslash S)/e^{-\langle z, \varepsilon'\xi_0 \rangle}\Omega^{\ttt}_{\PPP}(\V) &\xrightarrow{\sim} \varprojlim_{\varepsilon \to 0} H^1_{S_{\varepsilon}}(\V, e^{-\langle z, \varepsilon'\xi_0 \rangle}\Omega_{\PPP}^{\ttt})\\ &\xrightarrow{} H^1_{S_{\varepsilon}}(\V, e^{-\langle z, \varepsilon'\xi_0 \rangle}\Omega_{\PPP}^{\ttt})
\end{align*}
is given by $[\bar{\partial} (\psi_{\varepsilon} \omega)].$
\end{proposition}

\noindent
By analogy with Méril's spaces, we are led to introduce the following definitions :

\begin{definition}
For all $\varepsilon'>0$ we set
\[
\mathscr{H}^{\ttt}_{S}(\V,\varepsilon') = \frac{\{\omega \in \Omega_{\V}(\V \backslash S) : \forall r> \varepsilon > 0, \omega \in e^{-\langle z, \varepsilon'\xi_0 \rangle} \Db^{\ttt,1,0}_{\PPP}(S^{\circ}_r \backslash S_{\varepsilon})\}}{\{\omega \in \Omega_{\V}(\V) : \forall r>0, \omega \in e^{-\langle z, \varepsilon'\xi_0 \rangle} \Db^{\ttt,1,0}_{\PPP}(S^{\circ}_r)\}}.
\]
\end{definition}

Remark that $\mathscr{H}^{\ttt}_{S}(\V,\varepsilon') \simeq e^{-\langle z, \varepsilon'\xi_0 \rangle}\Omega^{\ttt_{\infty}}_{\PPP}(\V\backslash S)/e^{-\langle z, \varepsilon'\xi_0 \rangle}\Omega^{\ttt}_{\PPP}(\V)$ for all $\varepsilon'>0.$

\begin{definition}
For all $\varepsilon,\varepsilon'>0$ we set 
\medskip
\[
\text{Exp}^{\ttt}_{\varepsilon,\varepsilon'}(S) =e^{h_{S_{\varepsilon}}(w-\varepsilon'\xi_0)}\Hol^{\ttt}_{\PPP^*}((S_{\infty}^*)^{\circ}+\varepsilon'\xi_0)
\]
as well as
\[
\text{Exp}^{\ttt}_{\varepsilon'}(S) = \varprojlim_{\varepsilon\to 0} \text{Exp}^{\ttt}_{\varepsilon,\varepsilon'}(S), \quad
\text{Exp}^{\ttt}(S) =\varprojlim_{\varepsilon'\to 0} \text{Exp}^{\ttt}_{\varepsilon'}(S).
\] 
\end{definition}

\begin{theorem}\label{thm:maintheorem}
Let $\varepsilon'>0.$ There is a canonical isomorphism of $\C$-vector spaces
\begin{equation}
\mathscr{H}^{\ttt}_{S}(\V,\varepsilon') \xrightarrow{\sim} \Exp^{\ttt}_{\varepsilon'}(S). 
\end{equation}
These spaces do not depend on the chosen norm. 

\medskip
\noindent
Given a global $\C$-linear coordinate $z$ of $\V$ and $w$ its dual coordinate, this isomorphism can be explicited by
$
\mathscr{H}^{\ttt}_{S}(\V,\varepsilon') \ni [u(z)dz] \mapsto v \in \Exp^{\ttt}_{\varepsilon'}(S),
$
with
\[
v(w) = \int_{\partial S_{\varepsilon}^+} e^{zw} u(z)dz,
\] 
where $\partial S_{\varepsilon}^+$ is the positively oriented boundary of any thickening $S_{\varepsilon}$.
\end{theorem}
\begin{proof}
We apply $\underset{\varepsilon\to 0}\varprojlim$ to (\ref{equ:tempmeril1}) as well as (\ref{equ:tempmeril2}) to get isomorphisms
\[
\mathscr{H}^{\ttt}_{S}(\V,\varepsilon') \xrightarrow{\sim} \varprojlim_{\varepsilon\to 0}H^1_{S_{\varepsilon}}(\V, e^{-\langle z, \varepsilon'\xi_0 \rangle}\Omega_{\PPP}^{\ttt}) \xrightarrow{\sim} \varprojlim_{\varepsilon\to 0}\text{Exp}^{\ttt}_{\varepsilon,\varepsilon'}(S) = \text{Exp}^{\ttt}_{\varepsilon'}(S).
\]
\noindent
Let us now compute this map within coordinates. Let $[u(z)dz]\in \mathscr{H}^{\ttt}_{S}(\V,\varepsilon')$ and fix $\varepsilon>0.$ Let us choose a cutoff function $\psi_{\varepsilon}$ as in Proposition~\ref{prop:tempmeril3}. Then the image of $[u(z)dz]$ in $\text{Exp}^{\ttt}_{\varepsilon,\varepsilon'}(S)$ is given by $v$, where
\[
v(w) = \LL^+_w(\bar{\partial}(\psi_{\varepsilon}u(z)dz)) = \int_{\V} e^{zw} \bar{\partial} (\psi_{\varepsilon}u(z)dz).
\] One has 
\begin{align*}
\int_{\V} e^{zw} \bar{\partial} (\psi_{\varepsilon}u(z)dz) & = \int_{\V} \bar{\partial} (e^{zw}\psi_{\varepsilon}u(z)dz) = \int_{S_{\varepsilon}\backslash S^{\circ}_{\varepsilon/2}} \bar{\partial} (e^{zw}\psi_{\varepsilon}u(z)dz)\\
& = \lim_{R\to +\infty} \int_{(S_{\varepsilon}\backslash S^{\circ}_{\varepsilon/2})\cap \overline{D}(0,R)} \bar{\partial} (e^{zw}\psi_{\varepsilon}u(z)dz)\\
& = \lim_{R\to +\infty} \int_{\partial((S_{\varepsilon}\backslash S^{\circ}_{\varepsilon/2})\cap \overline{D}(0,R))^+}e^{zw}\psi_{\varepsilon}u(z)dz.
\end{align*}

\noindent
It is clear that $\partial((S_{\varepsilon}\backslash S^{\circ}_{\varepsilon/2})\cap \overline{D}(0,R))^+$ is a Jordan rectifiable curve which can be decomposed in four oriented rectifiable curves : $(\partial S_{\varepsilon} \cap \overline{D}(0,R))^+$, $(\partial S_{\varepsilon/2} \cap \overline{D}(0,R))^-$ and two oriented arcs of circle $\mathcal{I}_R$ and $\mathcal{J}_R$ (see figure \ref{contour} below). By construction of $\psi_{\varepsilon}$, we have
\[
\int_{(\partial S_{\varepsilon/2} \cap \overline{D}(0,R))^-} e^{zw}\psi_{\varepsilon}u(z)dz = 0
\] 
and 
\[
\lim_{R\to +\infty} \int_{(\partial S_{\varepsilon} \cap \overline{D}(0,R))^+} e^{zw}\psi_{\varepsilon}u(z)dz =\int_{\partial S_{\varepsilon}^+} e^{zw} u(z)dz.
\] 
Let us prove that 
\[
\lim_{R\to + \infty} \int_{\mathcal{I}_R} e^{zw}\psi_{\varepsilon}u(z)dz = \lim_{R\to + \infty} \int_{\mathcal{J}_R} e^{zw}\psi_{\varepsilon}u(z)dz =0.\] 
We do it for $\mathcal{I}_R.$ We have
\begin{align*}
\left|\int_{\mathcal{I}_R} e^{zw}\psi_{\varepsilon}(z)u(z) dz \right| &< 2\pi R \sup_{z\in \mathcal{I}_R} |e^{zw}u(z)|\\ &= 2\pi R \sup_{z\in \mathcal{I}_R} |e^{z(w-\varepsilon' \xi_0)}|\sup_{z\in \mathcal{I}_R}|e^{z(\varepsilon'\xi_0)}u(z)|.
\end{align*}

\bigskip

\begin{figure}
\centering
\begin{tikzpicture}[line cap=round,line join=round,>=stealth,x=1.0cm,y=1.0cm, scale = 1]
\tikzstyle{every node}=[scale=1];

\definecolor{uuuuuu}{rgb}{0.266,0.266,0.266}

\draw(0,0)  node[below left] {$0$} node {$\bullet$};

\clip(-4.,-3.5) rectangle (5.,5.);

\draw [samples=50,rotate around={-56.309:(1.44,1.833)},xshift=1.44cm,yshift=1.833cm,domain=-6.623:6.623)] plot (\x,{(\x)^2/2/1.103});
\draw [samples=50,rotate around={-56.309:(1.44,1.833)},xshift=1.44cm,yshift=1.233cm,domain=-6.623:6.623)] plot (\x,{(\x)^2/3/1.103});
\draw(0.,0.) circle (3.cm);

\draw [samples=50,rotate around={-56.309:(1.44,1.833)},xshift=1.44cm,yshift=1.833cm,domain=-0.85:1.25),line width=1.6pt] plot (\x,{(\x)^2/2/1.103});
\draw [samples=50,rotate around={-56.309:(1.44,1.833)},xshift=1.44cm,yshift=1.233cm,domain=-1.34:1.93),line width=1.6pt] plot (\x,{(\x)^2/3/1.103});
\draw [line width=1.6pt] (23.2:3) arc(23.2:10:3);  
\draw [line width=1.6pt] (65.6:3) arc(65.6:77.4:3);

\draw (44.4:2.29) node [rotate=134.4]{$>$};
\draw (44.4:1.69) node [rotate=134.4]{$<$};
\draw (17:3) node [rotate=108.5]{$>$};
\draw (71.5:3) node [rotate=150]{$>$};

\draw (2.7,-2.7) node [scale=1] {$C(0,R)$};
\draw (2.3,4) node [scale=1] {$\partial S_{\varepsilon/2}$};
\draw (0.5,4.03) node [scale=1] {$\partial S_{\varepsilon}$};
\draw (1,3.15) node [scale=0.8] {$\mathcal{I}_R$};
\draw (3.25,0.8) node [scale=0.8] {$\mathcal{J}_R$};
\end{tikzpicture}
\caption{\label{contour} The contour $\partial((S_{\varepsilon}\backslash S^{\circ}_{\varepsilon/2})\cap \overline{D}(0,R))^+$.}
\end{figure}

\bigskip
\noindent
On one hand, thanks to the tempered condition on $e^{z(\varepsilon'\xi_0)}u(z)dz$, one can see that, for $R$ big enough, there are $c\in (0,+\infty)$ and $N\in \N$ such that $\sup_{z\in \mathcal{I}_R}|e^{z(\varepsilon'\xi_0)}u(z)| \leq cR^N.$ On the other hand, for each $R>0$, there is $z_R \in \mathcal{I}_R$ such that $ \sup_{z\in \mathcal{I}_R} |e^{z(w-\varepsilon' \xi_0)}| = e^{\Re (z_R(w-\varepsilon'\xi_0))}.$ Moreover, one can write 
\[
 e^{\Re (z_R(w-\varepsilon'\xi_0))} = e^{|z_R||w-\varepsilon' \xi_0|\cos(\theta_R)} = e^{R|w-\varepsilon' \xi_0|\cos(\theta_R)},
\] 
where $\theta_R$ is the non-oriented angle between $\bar{z}_R$ and $w-\varepsilon' \xi_0.$ Since $w-\varepsilon' \xi_0 \in (S^*_{\infty})^{\circ}=(S^*_{\varepsilon, \infty})^{\circ}$ and $z_R \in S_{\varepsilon}$, we can find $\delta>0$ such that $\cos(\theta_R) < -\delta <0$ for all $R$ big enough. Hence, for $R$ big enough, 
\[
\left|\int_{\mathcal{I}_R} e^{zw}\psi_{\varepsilon}(z)u(z)\, dz \right|  < 2\pi c R^{N+1}e^{-|w-\varepsilon'\xi_0| \delta R} \underset{R\to +\infty}\to 0.
\] 
We have thus proved that the image of $[u(z)dz]$ in $\text{Exp}^{\ttt}_{\varepsilon,\varepsilon'}(S)$ is the function $v$, defined on $(S^*_{\infty})^{\circ}+\varepsilon'\xi_0$ by $v(w)=\int_{\partial S_{\varepsilon}^+} e^{zw} u(z)dz$. One can check, by a similar proof as above, that this integral remains unchanged with $\varepsilon_1<\varepsilon.$ Therefore, it is also the image of $[u(z)dz]$ in  $\text{Exp}^{\ttt}_{\varepsilon'}(S)$ and we get the conclusion.
\end{proof}

\begin{remark}
Let $\varepsilon' > \varepsilon'_1 >0.$ Then there is a well defined map $\mathscr{H}^{\ttt}_{S}(\V,\varepsilon'_1) \to \mathscr{H}^{\ttt}_{S}(\V,\varepsilon'),$ namely $[\omega] \mapsto [\omega].$ Indeed, if $e^{\langle z, \varepsilon'_1\xi_0 \rangle}\omega$ is tempered on $S^{\circ}_r \backslash S_{\varepsilon}$ (resp. on $S^{\circ}_r$), then 
\[
e^{\langle z, \varepsilon'\xi_0 \rangle}\omega = e^{\langle z, (\varepsilon'-\varepsilon'_1)\xi_0 \rangle}e^{\langle z, \varepsilon'_1\xi_0 \rangle}\omega
\] 
is also tempered on $S^{\circ}_r \backslash S_{\varepsilon}$ (resp. on $S^{\circ}_r$), since $\Re(\langle z, (\varepsilon'-\varepsilon'_1)\xi_0 \rangle) < 0$ for all $z\in S_r$ with big enough module. Hence, this gives rise to a projective system $(\mathscr{H}^{\ttt}_{S}(\V,\varepsilon'))_{\varepsilon'}$ which is compatible, through the Laplace transform, with the projective system $(\text{Exp}^{\ttt}_{\varepsilon'}(S))_{\varepsilon'}.$
\end{remark}

\begin{corollary}\label{cor:mainmeril}
There is a canonical isomorphism of $\C$-vector spaces
\begin{equation}\label{equ:mainmeril}
\varprojlim_{\varepsilon' \to 0}\mathscr{H}^{\ttt}_{S}(\V,\varepsilon') \xrightarrow{\sim} \Exp^{\ttt}(S). 
\end{equation}
\medskip
\noindent
Given a global $\C$-linear coordinate $z$ of $\V$ and $w$ its dual coordinate, this isomorphism can be explicited by
\[
\varprojlim_{\varepsilon' \to 0}\mathscr{H}^{\ttt}_{S}(\V,\varepsilon') \ni ([u_{\varepsilon'}(z)dz])_{\varepsilon'} \mapsto v \in \Exp^{\ttt}(S),
\]
with
\[
v(w) = \int_{\partial S_{\varepsilon}^+} e^{zw} u_{\varepsilon'}(z)dz.
\] 
\end{corollary}
\begin{proof}
Within coordinates, we already know that image of $([u_{\varepsilon'}(z)dz])_{\varepsilon'}$ through (\ref{equ:mainmeril}) is given by a family $(v_{\varepsilon'})_{\varepsilon'},$ where 
\[
v_{\varepsilon'}(w) = \int_{\partial S_{\varepsilon}^+} e^{zw} u_{\varepsilon'}(z)dz
\]
on $(S^*_{\infty})^{\circ}+ \varepsilon' \xi_0.$ To get the conclusion, it is enough to remark that
\begin{enumerate}
\item For all $\varepsilon'$, the function $v_{\varepsilon'}$ is well-defined and holomorphic on $(S^*_{\infty})^{\circ}.$
\item For any $\varepsilon' > \varepsilon'_1 > 0$, one has 
\[
\int_{\partial S_{\varepsilon}^+} e^{zw} u_{\varepsilon'}(z)dz = \int_{\partial S_{\varepsilon}^+} e^{zw} u_{\varepsilon_1'}(z)dz.\]
\end{enumerate}

\noindent
Indeed, since $u_{\varepsilon'}-u_{\varepsilon_1'}$ is entire and verifies a suitable tempered condition, we have

\[
\int_{\partial S_{\varepsilon}^+} e^{zw} (u_{\varepsilon'}(z)-u_{\varepsilon_1'}(z))dz =\!\! \lim_{R\to + \infty} \int_{\partial(S_{\varepsilon} \cap \overline{D}(0,R))^+} e^{zw} (u_{\varepsilon'}(z)-u_{\varepsilon_1'}(z))dz = 0.
\]
\end{proof}

\begin{remark}
Corollary~\ref{cor:mainmeril} is nothing more but Méril's theorem, while Theorem~\ref{thm:maintheorem} is a stronger and new result. Firstly, the canonical map
\[
\mathscr{H}_{S}(\V,\varepsilon') \to \mathscr{H}^{\ttt}_{S}(\V,\varepsilon') 
\]
is injective for all $\varepsilon'$, thanks to the Phragmen-Lindelöf theorem of \cite[p. 394]{Hill12}. Hence, it remains injective when applying $\underset{\varepsilon'\to 0}\varprojlim.$ Secondly, the inclusion 
\begin{align*}
\text{Exp}(S)&\subset \{v \in \Hol_{\V^*}((S_{\infty}^*)^{\circ}) : \forall \varepsilon,\varepsilon'>0, \, v \in e^{h_{S_{\varepsilon}}}\Db^{\ttt}_{\PPP}((S_{\infty}^*)^{\circ}+\varepsilon'\xi_0)\}\\ &\simeq \text{Exp}^{\ttt}(S)
\end{align*}
is an equality for the same reasons as in Remark~\ref{rem:polya}.
\end{remark}

\end{document}